\documentclass[a4paper,12pt]{amsart}
\usepackage[english]{babel}
\usepackage[T1]{fontenc}
\usepackage[left=1in,right=1in,top=1.2in,bottom=1.2in]{geometry}
\usepackage{times}

\usepackage[nopatch=footnote]{microtype}
\usepackage{lscape}		%for landscape pages
\usepackage{commath,amssymb,amscd,mathrsfs,mathtools,dsfont,bbm,multirow,array,caption,comment,pifont}
\usepackage{mathabx}
\usepackage[all]{xy}
\usepackage{enumitem}
\usepackage{longtable}
\usepackage{nicematrix}
\NiceMatrixOptions{xdots/horizontal-labels}
\usepackage{diagbox}
\usepackage{cite}			%won't show labels for citation

\allowdisplaybreaks

\usepackage[dvipsnames,svgnames,table]{xcolor}
\definecolor{red}{RGB}{255,25,25}
\definecolor{blue}{RGB}{25,50,200}
\usepackage{tikz-cd}
\usetikzlibrary{decorations.pathreplacing}

\usepackage{thmtools}   %load after amsthm and before cleveref to fix the \cref bug

\usepackage[pagebackref,linktocpage]{hyperref}

\hypersetup{
hypertexnames=false,
pdftitle={},				% title
pdfauthor={Hu-Jiang},	% author
colorlinks=true,			% false: boxed links; true: colored links
linkcolor=red,			% color of internal links (change box color with linkbordercolor)
citecolor=MidnightBlue,	% color of links to bibliography
filecolor=magenta,		% color of file links
urlcolor=red	% color of external links
}

\usepackage{cleveref}	% after hyperref!!

\newtheorem{theorem}{Theorem}[section]
\crefname{theorem}{Theorem}{Theorems}
\newtheorem{lemma}[theorem]{Lemma}
\crefname{lemma}{Lemma}{Lemmas}
\newtheorem{proposition}[theorem]{Proposition}
\crefname{proposition}{Proposition}{Propositions}

\crefname{prop}{Proposition}{Propositions}
\newtheorem{corollary}[theorem]{Corollary}
\crefname{corollary}{Corollary}{Corollaries}

\crefname{cor}{Corollary}{Corollaries}
\newtheorem{conjecture}[theorem]{Conjecture}
\crefname{conjecture}{Conjecture}{Conjectures}

\crefname{conj}{Conjecture}{Conjectures}
\newtheorem*{conj*}{Conjecture}
\crefname{conj}{Conjecture}{Conjectures}

\theoremstyle{definition}
\newtheorem{definition}[theorem]{Definition}
\crefname{definition}{Definition}{Definitions}

\crefname{defn}{Definition}{Definitions}

\crefname{example}{Example}{Examples}

\crefname{notation}{Notation}{Notation}
\newtheorem*{notation*}{Notation}
\crefname{notation}{Notation}{Notation}
\newtheorem*{convention*}{Convention}
\crefname{convention}{Convention}{Convention}

\crefname{problem}{Problem}{Problems}
\newtheorem{question}[theorem]{Question}
\crefname{question}{Question}{Questions}

\crefname{condition}{Condition}{Conditions}
\newtheorem{assumption}[theorem]{Assumption}
\crefname{assumption}{Assumption}{Assumptions}

\theoremstyle{remark}

\crefname{rmk}{Remark}{Remarks}
\newtheorem*{rmk*}{Remark}
\crefname{rmk}{Remark}{Remarks}
\newtheorem{remark}[theorem]{Remark}
\crefname{remark}{Remark}{Remarks}

\crefname{fact}{Fact}{Facts}
\newtheorem{claim}[theorem]{Claim}
\crefname{claim}{Claim}{Claims}
\newtheorem*{claim*}{Claim}
\crefname{claim}{Claim}{Claims}

\crefname{step}{Step}{Steps}

\crefname{case}{Case}{Cases}

\numberwithin{equation}{section}
\renewcommand{\emptyset}{\varnothing}

\newcommand{\doi}[1]{\href{https://doi.org/#1}{{\tt doi:#1}}}

\newcommand{\bC}{\mathbf{C}}

\newcommand{\bN}{\mathbf{N}}

\newcommand{\bQ}{\mathbf{Q}}
\newcommand{\bR}{\mathbf{R}}

\newcommand{\bZ}{\mathbf{Z}}

\newcommand{\bk}{\mathbf{k}}
\newcommand{\be}{\boldsymbol{e}}

\newcommand{\bt}{\boldsymbol{t}}

\newcommand{\bx}{\boldsymbol{x}}

\newcommand{\Aut}{\operatorname{Aut}}

\newcommand{\diff}{\mathrm{d}}

\newcommand{\id}{\operatorname{id}}

\newcommand{\lov}{\operatorname{lov}}

\newcommand{\N}{\mathsf{N}}

\newcommand{\num}{\equiv}
\newcommand{\nullity}{\operatorname{nullity}}
\newcommand{\wnum}{\equiv_{\mathsf{w}}}

\newcommand{\plov}{\operatorname{plov}}

\newcommand{\pr}{\operatorname{pr}}

\newcommand{\rank}{\operatorname{rank}}

\newcommand{\Vol}{\operatorname{Vol}}

\newcommand{\nullspace}{\operatorname{nullspace}}
\newcommand{\range}{\operatorname{range}}

\definecolor{purplefill}{HTML}{EEEDFE}
\definecolor{purpleval}{HTML}{CECBF6}
\definecolor{purpletext}{HTML}{26215C}
\definecolor{tealfill}{HTML}{E1F5EE}
\definecolor{tealval}{HTML}{9FE1CB}
\definecolor{tealtext}{HTML}{04342C}
\definecolor{amberfill}{HTML}{FAEEDA}
\definecolor{amberval}{HTML}{FAC775}
\definecolor{ambertext}{HTML}{412402}
\definecolor{grayfill}{HTML}{F1EFE8}
\definecolor{graytext}{HTML}{B4B2A9}

\newcommand{\pII}[1]{\cellcolor{purplefill}\color{purpletext}#1}
\newcommand{\pIIv}[1]{\cellcolor{purpleval}\color{purpletext}\textbf{#1}}
\newcommand{\tI}[1]{\cellcolor{tealfill}\color{tealtext}#1}
\newcommand{\tIv}[1]{\cellcolor{tealval}\color{tealtext}\textbf{#1}}
\newcommand{\aC}[1]{\cellcolor{amberfill}\color{ambertext}#1}
\newcommand{\aCv}[1]{\cellcolor{amberval}\color{ambertext}\textbf{#1}}

\newcommand{\zdot}{\cellcolor{grayfill}\color{graytext}$\cdot$}

\begin{document}

\title[A lower bound for $\mathrm{plov}$]{A lower bound for polynomial volume growth of automorphisms of zero entropy}

\author{Fei Hu}
\address{School of Mathematics, Nanjing University, Nanjing, China}
\email{\href{mailto:fhu@nju.edu.cn}{fhu@nju.edu.cn}}

\author{Chen Jiang}
\address{Shanghai Center for Mathematical Sciences \& School of Mathematical Sciences, Fudan University, Shanghai, China}
\email{\href{mailto:chenjiang@fudan.edu.cn}{chenjiang@fudan.edu.cn}}
 
\begin{abstract}
Let $X$ be a normal projective variety of dimension $d$, and let $f$ be a zero-entropy automorphism of $X$. Denote by $k$ the first-degree growth rate of $f$, so that $\deg_1(f^n) \asymp n^{k}$. We prove the sharp lower bound for the polynomial volume growth $\mathrm{plov}(f)$ of $f$:
\[
\mathrm{plov}(f) \ge d+\frac{k(k+2)}{4},
\]
equivalently giving a sharp lower bound on the Gelfand--Kirillov dimension of the associated twisted homogeneous coordinate ring. 
This improves previous lower bounds of Keeler and of Lin--Oguiso--Zhang.
In the proof, we introduce the notion of dynamical intersection polynomials and give a new characterization of $\mathrm{plov}(f)$ in terms of non-vanishing of intersection numbers.
We also establish a gap principle for polynomial volume growth: for every fixed dimension $d\ge 4$, either $\mathrm{plov}(f)=d^2$, or $\mathrm{plov}(f)\le d(d-2) + 2\lfloor d/4 \rfloor$.
This reveals a new rigidity phenomenon for zero-entropy automorphisms. As an application, in dimension $4$ we determine all possible values of $\mathrm{plov}$, thereby extending the results of Artin--Van den Bergh for surfaces and Lin--Oguiso--Zhang for threefolds.
\end{abstract}



\thanks{The first author was supported by NSFC Grant \#12371045.
The second author was supported by the National Key Research and Development Program of China \#2023YFA1010600 and NSFC for Innovative Research Groups \#12121001.}

\maketitle

\section{Introduction}

Throughout, unless otherwise stated, we work over an algebraically closed field $\bk$ of arbitrary characteristic. Let $X$ be a normal projective variety of dimension $d$, and let $H_X$ be an ample divisor on $X$.
Denote by $\bN$ the set of all nonnegative integers.

Given an endomorphism $f$ of $X$ and an integer $i$ with $1 \le i \le d$, the intersection number
$f^*H_X^i \cdot H_X^{d-i}$
is called the \textit{$i$-th degree} of $f$ with respect to $H_X$, and is denoted by $\deg_i(f)$.
The growth rate of the degree sequence $(\deg_i(f^n))_{n\ge 1}$ is closely related to the dynamical properties of the endomorphism $f$.
The \textit{$i$-th dynamical degree} $\lambda_i(f)$ of $f$ is defined by
\[
\lambda_i(f) \coloneqq \lim_{n\to\infty} \bigl(\deg_i(f^n)\bigr)^{1/n}
= \lim_{n\to\infty} \bigl((f^n)^*H_X^i \cdot H_X^{d-i}\bigr)^{1/n}.
\]
We say that $f$ has \textit{zero entropy} if the first dynamical degree $\lambda_1(f)$ is equal to $1$; equivalently, if all dynamical degrees $\lambda_i(f)$, for $i=1,\ldots,d$, are equal to $1$.
We refer to \cite{ES13,Truong20,Dang20,JL23,Truong25,Hu24,Xie-SC,Xie-GWRH} for fundamental properties of dynamical degrees and entropy in arbitrary characteristic.

For each integer $n\ge 2$, let $\Gamma_n \subseteq X^n$ denote the \textit{$n$-th iterate graph} of $f$, namely the graph of the morphism
$x \longmapsto (f(x),\ldots,f^{n-1}(x))$
from $X$ to $X^{n-1}$.
The growth rate of the volume of $\Gamma_n$ as $n\to\infty$ is another important dynamical invariant of $f$, called the \textit{volume growth} of $f$, and is defined by
\[
\lov(f) \coloneqq \limsup_{n\to\infty} \frac{\log \Vol(\Gamma_n)}{n}.
\]

When $X$ is a smooth complex projective variety, Gromov \cite{Gromov03} proved that the topological entropy $h_{\mathrm{top}}(f)$ of $f$ is bounded from above by the volume growth $\lov(f)$, which in turn is bounded from above by the logarithm of the maximal dynamical degree:
\[
h_{\mathrm{top}}(f) \le \lov(f) \le h_{\mathrm{alg}}(f) \coloneqq \log \max_{1\le i\le d} \lambda_i(f).
\]
On the other hand, Yomdin \cite{Yomdin87} proved the reverse inequality $h_{\mathrm{top}}(f) \ge h_{\mathrm{alg}}(f)$. Hence one has the following fundamental equality in higher-dimensional algebraic and holomorphic dynamics:
\[
h_{\mathrm{top}}(f)=\lov(f)=h_{\mathrm{alg}}(f).
\]

Recently, there has been growing interest in varieties with slow dynamics; see, for example, \cite{Cantat18-ICM,LB19,CPR21,FFO21,DLOZ22,LOZ25,Hu-GK-AV}.
Surprisingly, Lin--Oguiso--Zhang \cite{LOZ25} observed an interesting connection between automorphisms of zero entropy and noncommutative algebraic geometry; see also \cite{ATVdB90,AVdB90,Keeler00,RRZ06}.

From now on, we assume that $f$ is a zero-entropy automorphism of $X$.
Equivalently, there exists some $k\in 2\bN$ such that
\[
\deg_1(f^n) \coloneqq (f^n)^*H_X \cdot H_X^{d-1} \asymp n^k \qquad \text{as } n\to\infty.
\]
This is equivalent to saying that the pullback $f^*|_{\mathsf{N}^1(X)_\mathbf{R}}$ of $f$ on the real N\'eron--Severi space $\mathsf{N}^1(X)_\mathbf{R}$ is quasi-unipotent and the maximum size of the Jordan blocks of $f^*|_{\N^1(X)_\bR}$ is $k+1$.

The authors proved in \cite[Theorem~1.5]{HJ25} that $k\le 2d-2$, extending a previous result of \cite{DLOZ22} to arbitrary characteristic.
Following Cantat--Paris-Romaskevich \cite{CPR21}, the \textit{polynomial volume growth} $\plov(f)$ of $f$ is defined to be the polynomial growth rate of the volume of $\Gamma_n$ as $n\to\infty$, i.e.,
\[
\plov(f) \coloneqq \limsup_{n\to\infty} \frac{\log \Vol(\Gamma_n)}{\log n}.
\]
A useful fact due to Lin--Oguiso--Zhang states that $\plov(f)$ has the same parity as the dimension $d$ of $X$ (see \cite[Corollary~2.20]{LOZ25} or \cref{prop:same-parity}), which we call the \textit{First Gap Principle} for $\plov$.
In \cite[Theorem~1.1]{HJ25}, the authors also established an upper bound
\begin{equation}
\label{eq:plov-upper-bound}
\plov(f)\le (k/2+1)d,
\end{equation}
and hence, in particular, the optimal general upper bound
\[
\plov(f)\le d^2,
\]
which affirmatively answers questions of Cantat--Paris-Romaskevich \cite[Question~4.1]{CPR21} and Lin--Oguiso--Zhang \cite[Question~1.5(1)]{LOZ25}.

For the lower bound, Lin--Oguiso--Zhang \cite[Theorem~5.1]{LOZ25} showed that $\plov(f)\geq d+2k-2$, improving Keeler's lower bound $d+k$ \cite[Theorem~6.1(3)]{Keeler00} (see also \cite[Lemma~2.16]{LOZ25}). The first goal of this paper is to provide the optimal lower bound for $\plov(f)$ in terms of $d$ and $k$.

\begin{theorem}\label{thm:main1}
Let $X$ be a normal projective variety of dimension $d$ over $\bk$, and let $f$ be an automorphism of $X$ such that $\deg_1(f^n) \asymp n^k$ as $n\to\infty$.
Then we have the optimal lower bound
\begin{align*}
\plov(f) \ge d+\frac{k(k+2)}{4}.
\end{align*}
Moreover, one has
\[
k=2d-2 \iff \plov(f) = d^2.
\]
\end{theorem}

This theorem affirmatively answers an open question of the authors proposed in 2022; see \cite[Question~2.2 and Remark~2.3]{HJ25}.

\begin{remark}
In order to get the lower bound in \cref{thm:main1}, without loss of generality, we may assume that $f^*|_{\mathsf{N}^1(X)_\mathbf{R}}$ is unipotent.
We first introduce the so-called ``formal log-monodromy'' operator $L$ associated with $f$ (see \Cref{subsec:plov}), which also plays an important role in the proof of \cref{thm:k2d-4} later.
The next key ingredient is the \emph{dynamical intersection polynomial}
\[
\Phi_{f, H_X}(x_1,\ldots,x_d) \coloneqq \prod_{j=1}^d \bigl(\exp(x_j L)H_X\bigr) = \prod_{j=1}^d \, \sum_{i=0}^{k} \frac{L^iH_X}{i!} x_j^i \in \bR[x_1,\ldots,x_d]
\]
associated with $f$ and $H_X$.
It satisfies that for every $(m_1,\ldots, m_d)\in \bZ^d$, 
\[
\Phi_{f, H_X}(m_1,\ldots, m_d) = (f^{m_1})^*H_X\cdots (f^{m_d})^*H_X > 0.
\]
We show in \cref{thm:plov=d+max} via an elementary Riemann-sum asymptotic argument that
\[
\plov(f)=d+\deg \Phi_{f, H_X}.
\]
This interpretation reduces the calculation of $\plov(f)$ to the non-vanishing property of certain intersection numbers.
On the other hand, by the geometric positivity (see \cref{prop:thm3.2}), we readily obtain that the total degree $\deg \Phi_{f, H_X}$ of $\Phi_{f, H_X}$ is at least $k(k+2)/4$.
\end{remark}

From the perspective of classifying all possible values of $\plov$, equivalently of the Gelfand--Kirillov dimension (see \cref{rmk:GK-dim-gap}), it remains to consider the case in which $f$ does not have maximal first-degree growth, i.e., $k\le 2d-4$.
In this case, our second main theorem improves our previous upper bound for $\plov(f)$ from $d(d-1)$ to $d(d-2)+2\lfloor d/4\rfloor$ for $d\ge 4$.
Consequently, for any fixed dimension $d\ge 4$, the invariant $\plov$ cannot take any value in the interval $\bigl(d(d-2)+2\lfloor d/4\rfloor,\, d^2\bigr)$; see \cref{cor:gap}.
Motivated by this phenomenon, we later propose a more general gap principle for $\plov$; see \cref{conj:SGP}.

\begin{theorem}
\label{thm:k2d-4}
Let $X$ be a normal projective variety of dimension $d \ge 4$ over $\bk$, and let $f$ be a zero-entropy automorphism of $X$.
Suppose that the first-degree sequence $\deg_1(f^n)$ does not have maximal growth, i.e., $\deg_1(f^n)=O(n^{2d-4})$.
Then
\[
\plov(f)\le d(d-2)+2\lfloor d/4\rfloor
=
\begin{cases}
d(d-2)+\lfloor d/2\rfloor, & \text{if } d \equiv 0,1 \pmod 4; \\
d(d-2)+\lfloor d/2\rfloor-1, & \text{if } d \equiv 2,3 \pmod 4.
\end{cases}
\]
\end{theorem}

\begin{corollary}
\label{cor:gap}
Let $X$ be a normal projective variety of dimension $d\ge 2$ over $\bk$, and let $f$ be a zero-entropy automorphism of $X$. Then
\[
\plov(f)\notin \bigl(d(d-2)+2\max\bigl\{1, \, \lfloor d/4\rfloor\bigr\},\, d^2 \bigr).
\]
\end{corollary}

\begin{remark}
\label{rmk:GK-dim-gap}
As mentioned in \cite[Section~1.4]{LOZ25}, Keeler \cite{Keeler00} implicitly proved that the polynomial volume growth $\plov(f)$ is equal to the Gelfand--Kirillov dimension of the twisted homogeneous coordinate ring associated with $f$, minus one.
Therefore, our gap principle for $\plov$ also yields a gap principle for the Gelfand--Kirillov dimension of such twisted homogeneous coordinate rings arising from zero-entropy automorphisms (introduced in \cite{ATVdB90}).
Similar gap principles for the Gelfand--Kirillov dimension of various noncommutative algebras have been extensively studied since the 1970s; see, for example, \cite{Bass72,Guivarch73,Gromov81,SSW85,AS95,KL00,Smoktunowicz06,BRS10,Bell15} and the references therein.
\end{remark}

\begin{remark}
\label{rmk:idea-pf}
This paper complements and sharpens the authors' previous work \cite{HJ25}, which established an upper bound $(k/2+1)d$ for $\plov(f)$, but it may not be optimal when $k\le 2d-4$.
Here we prove the optimal lower bound for $\plov$ and identify a genuine gap in its possible values.
The proof in \cite{HJ25} reduces, using \cite[Lemma~2.16]{LOZ25}, to the vanishing of certain intersection numbers $v_\lambda$, where $\lambda\in P(k,d,n)$ ranges over the restricted partitions of $n$ (see \Cref{subsec:partition}).
An inductive argument shows that these intersection numbers $v_\lambda$ satisfy homogeneous linear systems whose coefficient matrices are the weighted incidence matrices $A_{k,d,n}$ when $n>dk/2$.
The required combinatorial vanishing, i.e., $v_\lambda=0$ for all $n>dk/2$, then follows from the full-rank property of these matrices; see \cref{thm:full-rank}.
\end{remark}

\begin{remark}
Let us now describe the main idea of the proof of \cref{thm:k2d-4}.
As mentioned earlier, we first introduce the ``formal log-monodromy operator'' $L$ associated with $f$, as well as intersection numbers $w_\lambda$ which are variants of $v_\lambda$; see \cref{eq:w-lambda}. The advantage is that these intersection numbers $w_\lambda$ with $\lambda \in P(k,d,n)$ satisfy genuine homogeneous linear equations with the same coefficient matrices $A_{k,d,n}$ for \textit{all} $1\leq n\leq dk$; see \cref{lemma:equations-v-lambda}.

By \cref{eq:plov-upper-bound}, it suffices to consider the submaximal case when $k=2d-4$.
By the parity principle, i.e., \cref{prop:same-parity}, it suffices to prove that
$\plov(f)\le d(d-2)+\lfloor d/2\rfloor$.
Then by \cref{thm:plov=d+max}, this further reduces to showing that the intersection numbers $w_\lambda$ vanish for all $\lambda\in P(2d-4,d,n)$ and all
\[
n\ge (d-1)(d-2)+\lfloor d/2\rfloor-1.
\]
In this case, \cref{prop:thm3.2} produces a distinguished restricted partition
\[
\kappa=(2d-4,2d-6,\ldots,2j_\circ,2j_\circ,\ldots,2,0) \in P(2d-4,d,(d-1)(d-2)+2j_\circ)
\]
for some $j_\circ$ with $0\le j_\circ\le d-2$.
Now, by the geometric vanishing, i.e., \cref{prop:thm3.2}\ref{prop:thm3.2-1}, all $w_\lambda$ with $\lambda \succ \kappa$ vanish.
For $w_\lambda$ with $\lambda \preceq \kappa$, we prove in \cref{thm:nullity-truncated-matrix2} by induction that the associated truncated matrix $A_{2d-4,d,n}^{\preceq \kappa}$ has trivial nullspace whenever
\[
n\ge (d-1)(d-2) + \max\bigl\{\, \lfloor d/2\rfloor-1, \, j_\circ+1\,\bigr\}.
\]
This lower bound on $n$ is precisely chosen so that our inductive argument works.
Finally, by combining this with the geometric positivity property, i.e., \cref{prop:thm3.2}\ref{prop:thm3.2-2}, we show that the above maximum is indeed $\lfloor d/2\rfloor-1$.
\end{remark}

As an application, in dimension $4$ we determine all possible values of $\plov$, equivalently of the Gelfand--Kirillov dimension, thereby extending the results of Artin--Van den Bergh for surfaces \cite[Theorem~1.7]{AVdB90} and Lin--Oguiso--Zhang for threefolds \cite[Corollary~1.3]{LOZ25}.

\begin{corollary}
\label{cor:explicit-values-lower-dim}
Let $X$ be a normal projective variety of dimension $4$ over $\bk$, and let $f$ be an automorphism of $X$ such that $\deg_1(f^n) \asymp n^k$ as $n\to\infty$. Then
\[
\plov(f)=
\begin{cases}
4 & \text{if } k=0; \\
6 \ \text{ or } \ 8 & \text{if } k=2; \\
10 & \text{if } k=4; \\
16 & \text{if } k=6.
\end{cases}
\]
All these values are realizable by $4$-dimensional abelian varieties. 
In particular, \cite[Question~1.5(2)]{LOZ25} has an affirmative answer in dimension $4$.
\end{corollary}

To conclude the introduction, we propose a gap conjecture concerning $\plov$ based on known results and lower-dimensional examples.

\begin{conjecture}[Second Gap Principle]
\label{conj:SGP}
Let $X$ be a normal projective variety of dimension $d \ge 2$ over $\bk$, and let $f$ be an automorphism of $X$ such that $\deg_1(f^n) \asymp n^k$ as $n\to\infty$.
Then
\[
k=2d-4 \iff \plov(f)=(d-1)^2+1;
\]
equivalently,
\[
\plov(f)\notin \bigl((d-1)^2+1,\, d^2\bigr).
\]
\end{conjecture}

\begin{remark}
\label{rmk:suspected-gap}
The equivalence between the two formulations of the conjecture follows from \cref{thm:main1} and the \cref{eq:plov-upper-bound}.
Note that when $k=2d-4$, \cref{thm:main1} already gives
\[
\plov(f)\ge (d-1)^2+1.
\]
So, \cref{conj:SGP} is also equivalent to the following (optimal) upper bound:
\[
k=2d-4 \implies \plov(f)\le (d-1)^2+1.
\]
Our \cref{cor:gap} has confirmed that the Second Gap Principle (\cref{conj:SGP}) holds for $d\le 7$.
However, we believe that to prove \cref{conj:SGP} in dimension $d\ge 8$ or to obtain sharper upper bounds for the polynomial volume growth when $k\le 2d-4$, it is inevitable to invoke more algebro-geometric tools.

More generally, for dimensions $d\ge 5$, an explicit description of all possible values of $\plov(f)$ remains wide open. See \cref{final question} for related discussions. 
\end{remark}

\subsection*{Acknowledgments}
The first author gratefully acknowledges the hospitality of the Fields Institute during Winter 2026, where part of this work was carried out. C.~Jiang is a member of the Key Laboratory of Mathematics for Nonlinear Sciences, Fudan University.

\section{Restricted partitions and weighted incidence matrices}
\label{sec:LA}

In this section, we recall restricted partitions and their associated weighted incidence matrices $A_{k,d,n}$.
In \Cref{subsec:partition-Akdn}, we introduce a natural block structure of $A_{k,d,n}$ based on the recurrence relation of restricted partition numbers.

\subsection{Restricted partitions}
\label{subsec:partition}

Fix positive integers $k,d\in \bZ_{>0}$.
For any $n\in \bZ$, a $d$-tuple $\lambda = (\lambda_1,\ldots,\lambda_d)$ of nonnegative integers is called a {\it restricted partition} of $n$ into at most $d$ parts with each part at most $k$, if
\[
|\lambda| \coloneqq \sum_{i=1}^d \lambda_i = n \quad \text{and} \quad k\ge \lambda_1\ge \lambda_2\ge \cdots \ge \lambda_d \ge 0.
\]
Occasionally, $\lambda$ is regarded as a vector in $\bN^d$ or $\bR^d$.
We denote by $P(k, d, n)$ the set of all such restricted partitions $\lambda$, and by $p(k,d,n)$ the cardinality of $P(k,d,n)$.
Note that if $n>dk$ or $n<0$, then $P(k, d, n)=\emptyset$ and 
$p(k,d,n)=0$.

Throughout, we adopt the \textit{lexicographic order} on restricted partitions, defined as follows. For two restricted partitions $\lambda, \mu \in \bN^d$, we write
\[\lambda \prec \mu,\]
if the leftmost nonzero entry of $\lambda-\mu$ is negative.
We also write $\lambda \preceq \mu$ if $\lambda = \mu$ or $\lambda \prec \mu$, and $\lambda \succ \mu$ if $\mu \prec \lambda$, and so on.
For each fixed $n$, we always list the elements of $P(k,d,n)$ in descending lexicographic order: for instance,
\[
P(4,3,6) = \{(4,2,0), (4,1,1), (3,3,0), (3,2,1), (2,2,2)\}.
\]

It is well known that the restricted partition numbers satisfy the following properties:
\begin{itemize}
\item Symmetry in $k$ and $d$: $p(k,d,n) = p(d,k,n)$;
\item Symmetry in $n$: $p(k,d,n) = p(k,d,dk-n)$;
\item Recurrence relation: $p(k,d,n) = p(k,d-1,n) + p(k-1,d,n-d)$;
\item Unimodality: 
\begin{align*}
\begin{cases}
p(k,d,n) \le p(k,d,n+1), &\text{if }  
n < {dk}/{2};\\
p(k,d,n) \ge p(k,d,n+1), &\text{if }  n \geq {dk}/{2}. 
\end{cases}
\end{align*}
\end{itemize}
See \cite[Chapter~3]{Andrews98} and \cite[Section~1.7]{Stanley97} for more details on restricted partitions.

It is often convenient to keep track of the multiplicity of each part.
Thus, we may alternatively write a restricted partition $\lambda=(\lambda_1,\ldots,\lambda_d)$ of $n$ in the form
\[
\lambda = (k^{e_k},\ldots,1^{e_1},0^{e_0}),
\]
where $e_i$ is the multiplicity of the part $i$.
Note that 
\[
e_i\in \bN, \ \sum_{i=0}^k e_i = d, \text{ and } \sum_{i=0}^k ie_i = n.
\]

\subsection{Weighted incidence matrices associated with restricted partitions}
\label{subsec:matrix}

Fix positive integers $k,d$.
Let $\mu= (k^{e_k},\ldots,1^{e_1},0^{e_0})\in P(k, d, n-1)$ be a restricted partition of $n-1\in \bN$.
For $ 0\le i\le k-1$, define
\[
\mu(i) \coloneqq (k^{e_k},\ldots,(i+1)^{e_{i+1}+1},i^{e_i-1},\ldots,0^{e_0}).
\]
Namely, $\mu(i)$ is obtained from $\mu$ by replacing one part equal to $i$ with $i+1$.
Then $\mu(i)\in P(k, d, n)$ is a restricted partition of $n$ whenever $e_i>0$.

\begin{definition}[Weighted incidence matrices]
\label{def:Akdn}
Let $n\in \bZ_{>0}$ be a positive integer.
For any $\mu\in P(k, d, n-1)$ and $\lambda\in P(k, d, n)$, we define
\[
a_{\mu, \lambda} \coloneqq
\begin{cases}
e_i, &\textrm{if } \mu= (k^{e_k},\ldots,1^{e_1},0^{e_0}) \textrm{ and } \lambda=\mu(i); \\
0, &\text{otherwise}.
\end{cases}
\]
Namely, $a_{\mu, \lambda}$ is the number of ways to obtain $\lambda$ from $\mu$ by increasing exactly one part by $1$ and then reordering if necessary.
We call
\[
A_{k, d, n} \coloneqq (a_{\mu, \lambda})_{\mu, \lambda} \in \bR^{p(k,d,n-1) \times p(k,d,n) }
\]
the {\it weighted incidence matrix} associated with $P(k, d, n-1)$ and $P(k, d, n)$.
\end{definition}

Below we give two examples of such matrices with $k=5$, $d=3$, and $n\in\{6,7\}$. These examples will also appear in Table~\ref{table:Akdn-block}.
\[
A_{5,3,6} = \begin{pmatrix}
2 & 0 & 0 & 0 & 0 & 0 \\
1 & 1 & 1 & 0 & 0 & 0 \\
0 & 1 & 0 & 1 & 1 & 0 \\
0 & 0 & 1 & 0 & 2 & 0 \\
0 & 0 & 0 & 0 & 2 & 1
\end{pmatrix},
\quad
A_{5,3,7} = \begin{pmatrix}
1 & 1 & 0 & 0 & 0 & 0 \\
1 & 0 & 1 & 1 & 0 & 0 \\
0 & 1 & 0 & 2 & 0 & 0 \\
0 & 0 & 2 & 0 & 1 & 0 \\
0 & 0 & 0 & 1 & 1 & 1 \\
0 & 0 & 0 & 0 & 0 & 3
\end{pmatrix}.
\]

The weighted incidence matrices $A_{k, d, n}$ defined in \cref{def:Akdn} are all of full rank.

\begin{theorem}[{cf. \cite[Theorem~4.1(2)]{HJ25}}]
\label{thm:full-rank}
Fix positive integers $k$ and $d$.
Then for every positive integer $n$ with $1\leq n\leq dk$, the matrix $A_{k, d, n}$ has full rank. More precisely,
\begin{align*}
\rank(A_{k, d, n}) = \min\{p(k, d, n-1), p(k, d, n)\} 
=
\begin{cases}
p(k, d, n-1), &\text{if } 1\leq n\leq \lceil dk/2 \rceil; \\
p(k, d, n), &\text{if } \lfloor dk/2 \rfloor < n\leq dk.
\end{cases}
\end{align*}
\end{theorem}

\subsection{Recurrence relation of weighted incidence matrices}
\label{subsec:partition-Akdn}

We discuss the set-theoretic relations among the sets $P(k,d,n)$ as $k,d$, and $n$ vary.
Note that for any $k'\leq k$, the set $P(k', d, n)$ is a subset of $P(k, d, n)$.
There is also a natural bijection between the set $P(k,d-1,n-k)$ of restricted partitions of $n-k$ and the subset of $P(k,d,n)$ consisting of those restricted partitions of $n$ whose first part is equal to $k$:
\[
\begin{aligned}
\iota_k \colon P(k,d-1,n-k) &\xrightarrow[\hspace*{0.5cm}]{\ 1\text{-}1\ } \{\lambda\in P(k,d,n) \,:\, \lambda_1=k\},\\
\nu = (\nu_1,\ldots,\nu_{d-1}) &\longmapsto \iota_k(\nu) \coloneqq (k,\nu) = (k,\nu_1,\ldots,\nu_{d-1}).
\end{aligned}
\]
This yields the following decomposition: for $d\ge 2$,
\begin{align}
\label{eq:decomposition-by-iota}
P(k,d,n)=\iota_k(P(k,d-1,n-k))\sqcup P(k-1,d,n),
\end{align}
and hence the recurrence relation
\begin{align*}
p(k,d,n) & = p(k,d-1,n-k) + p(k-1,d,n). 
\end{align*}

Accordingly, the following lemma gives the recurrence relation of $A_{k,d,n}$.

\begin{lemma}\label{lemma:Akdn-decomposition}
Under the decomposition \eqref{eq:decomposition-by-iota} for $P(k, d, n)$, and the analogous decomposition for $P(k, d, n-1)$, the weighted incidence matrix
$A_{k,d,n}$ has the following block lower-triangular form
\[\hspace{-2.5cm}
A_{k,d,n} = 
\begin{pNiceArray}{cc}[first-row,last-col,columns-width=1.5cm]
\Hbrace{1}{p(k,d-1,n-k)} & \Hbrace{1}{p(k-1,d,n)} \\
~~~A_{k,d-1,n-k}~~~ & ~~~\mathbf{0}~~~ & \Vbrace{1}{p(k,d-1,n-k-1)} \\[2pt]
* & A_{k-1,d,n} & \Vbrace{1}{p(k-1,d,n-1)} \\
\end{pNiceArray}.
\]
\end{lemma}
\begin{proof}
Let $\mu\in P(k,d,n-1)$ and $\lambda\in P(k,d,n)$ be arbitrary restricted partitions of $n-1$ and $n$, respectively.

First, suppose that
\[
\mu\in \iota_k(P(k,d-1,n-k-1)) \text{~~and~~} \lambda\in \iota_k(P(k,d-1,n-k)).
\]
Then the $(\mu, \lambda)$-entry of $A_{k,d,n}$ coincides with the $(\iota_k^{-1}(\mu), \iota_k^{-1}(\lambda))$-entry of $A_{k,d-1,n-k}$. Since $\iota_k$ is a bijection, these $(\iota_k^{-1}(\mu), \iota_k^{-1}(\lambda))$-entries fill the entire matrix $A_{k,d-1,n-k}$, which gives the top-left block of $A_{k,d,n}$.

Next, suppose that
\[
\mu\in P(k-1,d,n-1) \text{~~and~~} \lambda\in P(k-1,d,n).
\]
Then the $(\mu, \lambda)$-entry of $A_{k,d,n}$ coincides with the $(\mu, \lambda)$-entry of $A_{k-1,d,n}$, which is the bottom-right block of $A_{k,d,n}$.

Finally, suppose that
\[
\mu\in \iota_k(P(k,d-1,n-k-1)) \text{~~and~~} \lambda\in P(k-1,d,n).
\]
Then by the definition of $\iota_k$, the first part of $\mu$ is $k$, whereas the first part of $\lambda$ is at most $k-1$.
Hence $\lambda\neq \mu(i)$ for every $0\leq i\leq k-1$. 
It follows that the corresponding $(\mu, \lambda)$-entry of $A_{k,d,n}$ is zero, which gives the top-right block $\mathbf{0}$ of $A_{k,d,n}$.
\end{proof}

\begin{remark}
To state the recurrence relation of $A_{k,d,n}$ in \cref{lemma:Akdn-decomposition} clearly, we have to consider $A_{k,d,n}$ for $n\leq 0$ or $n>dk$. When $n=0$ (resp., $n=dk+1$), we regard $A_{k,d,n}$ as a ``matrix'' of size $0\times 1$ (resp., $1\times 0$), which has full row rank (resp., column rank). When $n<0$ or $n>dk+1$, we think of $A_{k,d,n}$ as a ``matrix'' of size $0\times 0$, which has both full row rank and full column rank. In particular, the block form in \cref{lemma:Akdn-decomposition} degenerates to $(*, A_{k-1,d,n})$ or $\binom{A_{k,d-1,n-k}}{*}$ when $n\leq k$ or $n>d(k-1)$, respectively.
This will not affect our discussion, since we are only interested in the nullity of $A_{k,d,n}$ (see \cref{lemma:nullity-lower-triangular}).
\end{remark}

To illustrate \cref{lemma:Akdn-decomposition}, Table~\ref{table:Akdn-block} below gives an explicit example of the block lower-triangular form of $A_{6,4,12}$. 
The purple block corresponds to $A_{6,3,6}$, and the teal block corresponds to $A_{5,4,12}$. Moreover, the matrix $A_{5,3,6}$ (resp. $A_{5,3,7}$) appears as the bottom-right block of the purple block $A_{6,3,6}$ (resp. as the top-left block of the teal block $A_{5,4,12}$); these subblocks are indicated by dashed lines.

\begin{table}[htbp]
\footnotesize
\setlength{\tabcolsep}{5pt}
\caption{Block lower-triangular form of $A_{6,4,12}$}
\label{table:Akdn-block}

\begin{NiceTabular}{r|cccccccccccccccccc}
  & \cellcolor{purplefill}\rotatebox{90}{\tiny\color{purpletext}(6,6,0,0)} & \cellcolor{purplefill}\rotatebox{90}{\tiny\color{purpletext}(6,5,1,0)} & \cellcolor{purplefill}\rotatebox{90}{\tiny\color{purpletext}(6,4,2,0)} & \cellcolor{purplefill}\rotatebox{90}{\tiny\color{purpletext}(6,4,1,1)} & \cellcolor{purplefill}\rotatebox{90}{\tiny\color{purpletext}(6,3,3,0)} & \cellcolor{purplefill}\rotatebox{90}{\tiny\color{purpletext}(6,3,2,1)} & \cellcolor{purplefill}\rotatebox{90}{\tiny\color{purpletext}(6,2,2,2)} & \cellcolor{tealfill}\rotatebox{90}{\tiny\color{tealtext}(5,5,2,0)} & \cellcolor{tealfill}\rotatebox{90}{\tiny\color{tealtext}(5,5,1,1)} & \cellcolor{tealfill}\rotatebox{90}{\tiny\color{tealtext}(5,4,3,0)} & \cellcolor{tealfill}\rotatebox{90}{\tiny\color{tealtext}(5,4,2,1)} & \cellcolor{tealfill}\rotatebox{90}{\tiny\color{tealtext}(5,3,3,1)} & \cellcolor{tealfill}\rotatebox{90}{\tiny\color{tealtext}(5,3,2,2)} & \cellcolor{tealfill}\rotatebox{90}{\tiny\color{tealtext}(4,4,4,0)} & \cellcolor{tealfill}\rotatebox{90}{\tiny\color{tealtext}(4,4,3,1)} & \cellcolor{tealfill}\rotatebox{90}{\tiny\color{tealtext}(4,4,2,2)} & \cellcolor{tealfill}\rotatebox{90}{\tiny\color{tealtext}(4,3,3,2)} & \cellcolor{tealfill}\rotatebox{90}{\tiny\color{tealtext}(3,3,3,3)} \\
\hline
  \cellcolor{purplefill}\tiny\color{purpletext}(6,5,0,0) & \pIIv{1} & \pIIv{2} & \pII{$\cdot$} & \pII{$\cdot$} & \pII{$\cdot$} & \pII{$\cdot$} & \pII{$\cdot$} & \zdot & \zdot & \zdot & \zdot & \zdot & \zdot & \zdot & \zdot & \zdot & \zdot & \zdot \\
  \cellcolor{purplefill}\tiny\color{purpletext}(6,4,1,0) & \pII{$\cdot$} & \pIIv{1} & \pIIv{1} & \pIIv{1} & \pII{$\cdot$} & \pII{$\cdot$} & \pII{$\cdot$} & \zdot & \zdot & \zdot & \zdot & \zdot & \zdot & \zdot & \zdot & \zdot & \zdot & \zdot \\
  \cellcolor{purplefill}\tiny\color{purpletext}(6,3,2,0) & \pII{$\cdot$} & \pII{$\cdot$} & \pIIv{1} & \pII{$\cdot$} & \pIIv{1} & \pIIv{1} & \pII{$\cdot$} & \zdot & \zdot & \zdot & \zdot & \zdot & \zdot & \zdot & \zdot & \zdot & \zdot & \zdot \\
  \cellcolor{purplefill}\tiny\color{purpletext}(6,3,1,1) & \pII{$\cdot$} & \pII{$\cdot$} & \pII{$\cdot$} & \pIIv{1} & \pII{$\cdot$} & \pIIv{2} & \pII{$\cdot$} & \zdot & \zdot & \zdot & \zdot & \zdot & \zdot & \zdot & \zdot & \zdot & \zdot & \zdot \\
  \cellcolor{purplefill}\tiny\color{purpletext}(6,2,2,1) & \pII{$\cdot$} & \pII{$\cdot$} & \pII{$\cdot$} & \pII{$\cdot$} & \pII{$\cdot$} & \pIIv{2} & \pIIv{1} & \zdot & \zdot & \zdot & \zdot & \zdot & \zdot & \zdot & \zdot & \zdot & \zdot & \zdot \\
  \cellcolor{tealfill}\tiny\color{tealtext}(5,5,1,0) & \aC{$\cdot$} & \aCv{2} & \aC{$\cdot$} & \aC{$\cdot$} & \aC{$\cdot$} & \aC{$\cdot$} & \aC{$\cdot$} & \tIv{1} & \tIv{1} & \tI{$\cdot$} & \tI{$\cdot$} & \tI{$\cdot$} & \tI{$\cdot$} & \tI{$\cdot$} & \tI{$\cdot$} & \tI{$\cdot$} & \tI{$\cdot$} & \tI{$\cdot$} \\
  \cellcolor{tealfill}\tiny\color{tealtext}(5,4,2,0) & \aC{$\cdot$} & \aC{$\cdot$} & \aCv{1} & \aC{$\cdot$} & \aC{$\cdot$} & \aC{$\cdot$} & \aC{$\cdot$} & \tIv{1} & \tI{$\cdot$} & \tIv{1} & \tIv{1} & \tI{$\cdot$} & \tI{$\cdot$} & \tI{$\cdot$} & \tI{$\cdot$} & \tI{$\cdot$} & \tI{$\cdot$} & \tI{$\cdot$} \\
  \cellcolor{tealfill}\tiny\color{tealtext}(5,4,1,1) & \aC{$\cdot$} & \aC{$\cdot$} & \aC{$\cdot$} & \aCv{1} & \aC{$\cdot$} & \aC{$\cdot$} & \aC{$\cdot$} & \tI{$\cdot$} & \tIv{1} & \tI{$\cdot$} & \tIv{2} & \tI{$\cdot$} & \tI{$\cdot$} & \tI{$\cdot$} & \tI{$\cdot$} & \tI{$\cdot$} & \tI{$\cdot$} & \tI{$\cdot$} \\
  \cellcolor{tealfill}\tiny\color{tealtext}(5,3,3,0) & \aC{$\cdot$} & \aC{$\cdot$} & \aC{$\cdot$} & \aC{$\cdot$} & \aCv{1} & \aC{$\cdot$} & \aC{$\cdot$} & \tI{$\cdot$} & \tI{$\cdot$} & \tIv{2} & \tI{$\cdot$} & \tIv{1} & \tI{$\cdot$} & \tI{$\cdot$} & \tI{$\cdot$} & \tI{$\cdot$} & \tI{$\cdot$} & \tI{$\cdot$} \\
  \cellcolor{tealfill}\tiny\color{tealtext}(5,3,2,1) & \aC{$\cdot$} & \aC{$\cdot$} & \aC{$\cdot$} & \aC{$\cdot$} & \aC{$\cdot$} & \aCv{1} & \aC{$\cdot$} & \tI{$\cdot$} & \tI{$\cdot$} & \tI{$\cdot$} & \tIv{1} & \tIv{1} & \tIv{1} & \tI{$\cdot$} & \tI{$\cdot$} & \tI{$\cdot$} & \tI{$\cdot$} & \tI{$\cdot$} \\
  \cellcolor{tealfill}\tiny\color{tealtext}(5,2,2,2) & \aC{$\cdot$} & \aC{$\cdot$} & \aC{$\cdot$} & \aC{$\cdot$} & \aC{$\cdot$} & \aC{$\cdot$} & \aCv{1} & \tI{$\cdot$} & \tI{$\cdot$} & \tI{$\cdot$} & \tI{$\cdot$} & \tI{$\cdot$} & \tIv{3} & \tI{$\cdot$} & \tI{$\cdot$} & \tI{$\cdot$} & \tI{$\cdot$} & \tI{$\cdot$} \\
  \cellcolor{tealfill}\tiny\color{tealtext}(4,4,3,0) & \aC{$\cdot$} & \aC{$\cdot$} & \aC{$\cdot$} & \aC{$\cdot$} & \aC{$\cdot$} & \aC{$\cdot$} & \aC{$\cdot$} & \tI{$\cdot$} & \tI{$\cdot$} & \tIv{2} & \tI{$\cdot$} & \tI{$\cdot$} & \tI{$\cdot$} & \tIv{1} & \tIv{1} & \tI{$\cdot$} & \tI{$\cdot$} & \tI{$\cdot$} \\
  \cellcolor{tealfill}\tiny\color{tealtext}(4,4,2,1) & \aC{$\cdot$} & \aC{$\cdot$} & \aC{$\cdot$} & \aC{$\cdot$} & \aC{$\cdot$} & \aC{$\cdot$} & \aC{$\cdot$} & \tI{$\cdot$} & \tI{$\cdot$} & \tI{$\cdot$} & \tIv{2} & \tI{$\cdot$} & \tI{$\cdot$} & \tI{$\cdot$} & \tIv{1} & \tIv{1} & \tI{$\cdot$} & \tI{$\cdot$} \\
  \cellcolor{tealfill}\tiny\color{tealtext}(4,3,3,1) & \aC{$\cdot$} & \aC{$\cdot$} & \aC{$\cdot$} & \aC{$\cdot$} & \aC{$\cdot$} & \aC{$\cdot$} & \aC{$\cdot$} & \tI{$\cdot$} & \tI{$\cdot$} & \tI{$\cdot$} & \tI{$\cdot$} & \tIv{1} & \tI{$\cdot$} & \tI{$\cdot$} & \tIv{2} & \tI{$\cdot$} & \tIv{1} & \tI{$\cdot$} \\
  \cellcolor{tealfill}\tiny\color{tealtext}(4,3,2,2) & \aC{$\cdot$} & \aC{$\cdot$} & \aC{$\cdot$} & \aC{$\cdot$} & \aC{$\cdot$} & \aC{$\cdot$} & \aC{$\cdot$} & \tI{$\cdot$} & \tI{$\cdot$} & \tI{$\cdot$} & \tI{$\cdot$} & \tI{$\cdot$} & \tIv{1} & \tI{$\cdot$} & \tI{$\cdot$} & \tIv{1} & \tIv{2} & \tI{$\cdot$} \\
  \cellcolor{tealfill}\tiny\color{tealtext}(3,3,3,2) & \aC{$\cdot$} & \aC{$\cdot$} & \aC{$\cdot$} & \aC{$\cdot$} & \aC{$\cdot$} & \aC{$\cdot$} & \aC{$\cdot$} & \tI{$\cdot$} & \tI{$\cdot$} & \tI{$\cdot$} & \tI{$\cdot$} & \tI{$\cdot$} & \tI{$\cdot$} & \tI{$\cdot$} & \tI{$\cdot$} & \tI{$\cdot$} & \tIv{3} & \tIv{1} \\

\CodeAfter
  \tikz \draw[dashed, line width=0.6pt]
    ([xshift=-5.8pt,yshift=3.8pt]2-3.north west)
    rectangle
    ([xshift=5.8pt,yshift=-4pt]6-8.south east);

  \tikz \draw[dashed, line width=0.6pt]
    ([xshift=-5.8pt,yshift=4pt]7-9.north west)
    rectangle
    ([xshift=5.8pt,yshift=-4pt]12-14.south east);
\end{NiceTabular}

\vspace{6pt}
\noindent
{\color{purpleval}\rule{7pt}{7pt}}\,$A_{6,3,6}$\;\;
{\color{graytext}\rule{7pt}{7pt}}\,$\mathbf{0}$\;\;
{\color{amberval}\rule{7pt}{7pt}}\,$*$\;\;
{\color{tealval}\rule{7pt}{7pt}}\,$A_{5,4,12}$
\end{table}

\section{Polynomial volume growth and intersection numbers}
\label{sec:AG}

Throughout this section, we work over an algebraically closed field $\bk$ of arbitrary characteristic.
Let $X$ be a normal projective variety of dimension $d$, and let $H_X$ be an ample divisor on $X$.
Let $f$ be an automorphism of $X$ of zero entropy, and let $k\geq 0$ be the nonnegative integer such that
\[
\deg_1(f^n) \coloneqq (f^n)^*H_X \cdot H_X^{d-1} \asymp n^k, \quad \text{as } n\to\infty.
\]
It is known that $k$ is even by \cite[Lemma~6.12]{Keeler00} and satisfies $k\leq 2d-2$ by \cite[Theorem~1.5]{HJ25}; see also \cite[Theorem~1.1]{DLOZ22}.
Equivalently, the action of $f^*$ on the real N\'eron--Severi space $\N^1(X)_\bR$ is quasi-unipotent, and the maximum size of the Jordan blocks in the Jordan canonical form of $f^*|_{\N^1(X)_\bR}$ is $k+1$.

In the study of the polynomial volume growth and the degree growth of a zero-entropy automorphism $f$, it is harmless to replace $f$ by an arbitrary iterate $f^n$ with $n\in \bZ\setminus \{0\}$; see, e.g., \cite[Lemma~2.6]{Hu-GK-AV}.
We henceforth make the following assumption throughout \Cref{sec:AG}.

\begin{assumption}\label{assum:XfN}
Let $X$ be a normal projective variety of dimension $d\ge 2$ over $\bk$, and let $H_X$ be an ample divisor on $X$.
Let $f$ be an automorphism of $X$ such that 
\begin{equation}
\label{eq:f=1+N}
f^*|_{\N^1(X)_\bR} = \id + N,
\end{equation}
where $N$ is nilpotent with nilpotency index $k+1$, i.e., $N^{k+1} = 0$ but $N^k \neq 0$, and $k=2r$ is a positive even integer with $k\le 2d-2$.
\end{assumption}

\subsection{Polynomial volume growth and formal log-monodromy}
\label{subsec:plov}

The interpretation \eqref{eq:f=1+N} plays a crucial role in previous studies of the polynomial volume growth of $f$.
Indeed, using this interpretation, the function $n\mapsto \Vol(\Gamma_n)$ is a polynomial in $n$, whose degree is exactly $\plov(f)$; see \cite[Lemma~2.16]{LOZ25}.
In the proof of \cite[Theorem~1.1]{HJ25}, the authors also rely on this interpretation and reduce the problem of bounding $\plov(f)$ from above to proving the vanishing of the intersection numbers
$N^{\lambda_1}H_X \cdots N^{\lambda_d} H_X$
for $\lambda \in P(k,d,n)$ with $n>dk/2$.
By an inductive argument, they show that these intersection numbers satisfy certain homogeneous linear equations encoded by the weighted incidence matrices $A_{k,d,n}$ defined in \cref{def:Akdn}.
For the purpose of deriving lower bounds for $\plov(f)$, however, this interpretation has the disadvantage that these homogeneous linear equations make sense only when $n>dk/2$, besides the potential cancellation.

To overcome this issue, we introduce the ``formal log-monodromy'' operator $L$ that yields genuine homogeneous linear equations on the intersection numbers $w_\lambda$ for all $\lambda \in P(k,d,n)$ and all $1\leq n\leq dk$ (see \cref{lemma:equations-v-lambda}), and moreover, circumvents the potential cancellation (see \cref{thm:plov=d+max,rmk:cancellation}).
We define the operator $L$ by
\begin{equation}
\label{eq:L=logf}
L \coloneqq \log(f^*|_{\N^1(X)_\bR}) = \sum_{i=1}^{k} \frac{(-1)^{i+1}}{i} N^i.
\end{equation}
Then $L$ is also a nilpotent endomorphism of $\N^1(X)_\bR$ with nilpotency index $k+1$, and we have
\[
f^*|_{\N^1(X)_\bR} = \exp(L) = \sum_{i=0}^{k} \frac{L^i}{i!} \quad \text{and} \quad N = \sum_{i=1}^{k} \frac{L^i}{i!}.
\]
We will see in \Cref{subsec:intersection-numbers} that $L$ keeps many good properties of $N$.

Recall that by the projection formula, the volume $ \Vol(\Gamma_n)$ of $\Gamma_n$ against the ample divisor $\sum_{i=1}^{n}\pr_i^*H_X$ on $X^n$ can be expressed as the top self-intersection of the ample divisor $\Delta_n(f, H_X)$ on $X$, defined by
\[
\Delta_n(f, H_X) \coloneqq \sum_{m=0}^{n-1} (f^m)^*H_X.
\]
Now using the interpretation \eqref{eq:L=logf}, we can express $\Delta_n(f, H_X)$ as a polynomial function in $n$ with coefficients in $\N^1(X)_\bR$ as follows:
\[
\Delta_n(f, H_X) = \sum_{m=0}^{n-1} \exp(mL) H_X = \sum_{m=0}^{n-1} \, \sum_{i=0}^{k} \frac{m^i}{i!} L^i H_X = \sum_{i=0}^{k} \frac{S_i(n-1)}{i!} L^i H_X,
\]
where
\[
S_i(n-1) = \sum_{m=0}^{n-1} m^i = \frac{B_{i+1}(n) - B_{i+1}(0)}{i+1}.
\]
Note that the $(i+1)$-th Bernoulli polynomial $B_{i+1}(n)$ is monic of degree $i+1$. Then it is clear that the intersection number $\Delta_n(f, H_X)^d$ is a polynomial in $n$ with intersection numbers of $\{L^i H_X\}_{0\leq i\leq k}$ as coefficients. 

The following interpretation of $\plov$ is essential in the previous works \cite{Keeler00,LOZ25,Hu-GK-AV,HJ25}.
See also \cite[Proof of Lemma~6.13]{Keeler00}.

\begin{lemma}[{cf.~\cite[Lemma~2.16]{LOZ25}}]
\label{lemma:plov}
Assume the setting and the notation of \cref{assum:XfN}.
Then $\plov(f)$ is equal to the degree of the polynomial $\Delta_n(f, H_X)^d$ in $n$.
\end{lemma}

We recall the following results on the upper bounds for $\plov(f)$ and the first-degree growth rate $k$ of zero-entropy automorphisms $f$.
Note that combining them yields the optimal upper bound $\plov(f)\le d^2$.

\begin{theorem}[{see \cite[Theorems~1.1 and 1.5]{HJ25}}]
\label{thm:upper-bounds}
Let $X$ be a normal projective variety of dimension $d$ over $\bk$, and let $f$ be an automorphism of $X$ such that $\deg_1(f^n) \asymp n^k$ as $n\to\infty$.
Then we have
\[
\plov(f)\le (k/2+1)d \quad \text{and} \quad k\le 2d-2.
\]
\end{theorem}

As a useful consequence of \cref{lemma:plov}, Lin, Oguiso, and Zhang proved that the polynomial volume growth $\plov(f)$ has the same parity as the dimension of $X$; see \cite[Corollary~2.20]{LOZ25}.
This can be viewed as the \textit{First Gap Principle} for $\plov$.
We also refer to it as the \textit{parity principle} for $\plov$.

\begin{proposition}[Parity principle]
\label{prop:same-parity}
One has that $\plov(f) \equiv d \pmod 2$.
\end{proposition}

\begin{remark}
\label{rmk:odd-d}
When $d$ is odd and $k \equiv 2 \pmod 4$, combining \cref{prop:same-parity} with the inequality $\plov(f) \leq (k/2+1)d$ (see \cref{thm:upper-bounds}), we obtain that $\plov(f) \leq (k/2+1)d-1$, which improves \cite[Proposition~5.4]{HJ25}.
In particular, if $k=2d-4$, then one has
\[
\plov(f) \le d(d-2) + 2\lfloor d/2\rfloor.
\]
Our \cref{thm:k2d-4} gives a further improvement of this upper bound with $\lfloor d/2\rfloor$ replaced by $\lfloor d/4\rfloor$, while the Second Gap Principle (\cref{conj:SGP}) predicts that one can even replace $\lfloor d/2\rfloor$ by $1$.
\end{remark}

Recall that by \cite{LOZ25} (see also \cite{Hu-GK-AV} in positive characteristic), we now know all realizable values of $\plov$ for abelian varieties.

\begin{lemma}[{cf.~\cite[Theorem~1.4]{LOZ25} and \cite[Theorem~1.1]{Hu-GK-AV}}]
\label{lem:possible-plov}
For every pair of integers $d\in \bZ_{>0}$ and $k\in 2\bZ\cap [0,2d-2]$, 
the following two sets coincide:
\[
\Biggl\{\, \plov(f) \,:\,
\begin{aligned}
&  X \text{ is an abelian variety}, \, \dim X= d,\\
& f\in \Aut(X), \, \deg_1(f^n) \asymp n^k
\end{aligned}
\,\Biggr\}
=
\Biggl\{\,\sum_{i=1}^{d} \mu_i^2 \,:\, 
\begin{aligned}
& \mu \in P(d,d,d), \\
& \mu_1=k/2+1
\end{aligned}
\,\Biggr\}.
\]
\end{lemma}
\begin{proof}
Note that for any zero-entropy automorphism $f$ of an abelian variety $X$, one has that $\deg_1(f^n) \asymp n^k$ if and only if the maximum size of Jordan blocks of $f^*$ on $H^1(X, \bC)$ (or on the $\ell$-adic Tate space $V_\ell(X)$, in positive characteristic) is $k/2+1$, see \cite[Theorem~1.12, Remark~1.13]{Hu-GK-AV}.
Hence by \cite[Theorem~1.4]{LOZ25} and \cite[Theorem~1.1]{Hu-GK-AV}, the left-hand set is contained in the right-hand set.

Conversely, all values in the right-hand set are realizable by a standard construction: for any $\mu\in P(d,d,d)$ with $\mu_1=k/2+1$, one can take $X=E^d$ for an elliptic curve $E$, and let $f$ be the automorphism of $X$ defined coordinatewise by the unipotent matrix $\bigoplus_{i=1}^d J_{1,\mu_i}$, omitting the zero parts.
Then again by \cite[Theorem~1.4]{LOZ25} and \cite[Theorem~1.1]{Hu-GK-AV}, we have that $\plov(f)$ is exactly $\sum_{i=1}^d \mu_i^2$ (cf.~\cite[Example~5.6]{HJ25}).
This proves the reverse inclusion.
\end{proof}

\subsection{Dynamical intersection polynomials}
\label{subsec:mixed-poly}

One key contribution of this paper is the new characterization of $\plov$ in terms of the total degree of the associated dynamical intersection polynomial; see \cref{thm:plov=d+max}.
Recall that under \cref{assum:XfN}, the ``formal log-monodromy'' operator $L$ on $\N^1(X)_\bR$ is defined by \cref{eq:L=logf}, so that
\[
f^*|_{\N^1(X)_\bR} = \exp(L) = \sum_{i=0}^{k} \frac{L^i}{i!}.
\]

\begin{definition}[Dynamical intersection polynomials]
\label{def:Phi}
Under \cref{assum:XfN}, the \emph{dynamical intersection polynomial} associated with $f$ and $H_X$ is defined by the real symmetric polynomial
\begin{align}\label{eq:def Phi}
\begin{split}
\Phi_{f, H_X}(\bx) &\coloneqq \prod_{j=1}^d \bigl(\exp(x_j L)H_X\bigr) = \prod_{j=1}^d \, \sum_{i=0}^{k} \frac{L^iH_X}{i!} x_j^i \\[2pt]
&=\sum_{0\le \lambda_1,\ldots, \lambda_d \le k} \frac{w_\lambda}{\lambda!} \, \bx^\lambda \in \bR[x_1,\ldots,x_d],
\end{split}
\end{align}
where for each $\lambda = (\lambda_1,\ldots,\lambda_d) \in \{0,1,\ldots,k\}^d$, we set $\lambda!\coloneqq \lambda_1!\cdots\lambda_d!$,
\[
w_\lambda \coloneqq (L^{\lambda_1}H_X)\cdots (L^{\lambda_d}H_X), \quad \text{and} \quad \bx^\lambda \coloneqq x_1^{\lambda_1}\cdots x_d^{\lambda_d}.
\]
Denote by $\Phi_{f, H_X}^{\mathrm{top}}(\bx)$ the highest-degree homogeneous part of $\Phi_{f, H_X}(\bx)$; equivalently, it is just the leading coefficient of  $\Phi_{f, H_X}(t\bx)$ as a polynomial in $t$.
\end{definition}

As intersection product is symmetric, it suffices to consider the intersection numbers $w_\lambda$ with
\[k\geq \lambda_1\geq \cdots \geq \lambda_d\geq 0.\]
This is exactly where restricted partitions enter the study of $\plov$.
More precisely, for each integer $n$ with $0\leq n\leq dk$ and each restricted partition
\[
\lambda = (\lambda_1,\ldots,\lambda_d) = (k^{e_k},\ldots,1^{e_1},0^{e_0}) \in P(k,d,n),
\]
we write
\begin{equation}
\label{eq:w-lambda}
w_\lambda = (L^{\lambda_1}H_X) \cdots (L^{\lambda_d}H_X) = \prod_{i=0}^k \, (L^{i}H_X)^{e_i}.
\end{equation}

Note that by definition, for every $(m_1,\ldots, m_d)\in \bZ^d$, one has
\begin{align}\label{eq Phi mi >0}
\Phi_{f, H_X}(m_1,\ldots, m_d) = \prod_{j=1}^d \bigl(\exp(m_j L)H_X\bigr) = \prod_{j=1}^d \bigl((f^{m_j})^*H_X\bigr) > 0,
\end{align}
since each $(f^{m_j})^*H_X$ is ample.
Also, by the projection formula, $\Phi_{f, H_X}(m,\ldots,m)=H_X^d>0$ for every $m\in \bZ$.
Hence $\Phi_{f, H_X}^{\mathrm{top}}$ is well defined and nonzero. We next show that $\Phi_{f, H_X}^{\mathrm{top}}$ is nonnegative on $\bR^d$.

\begin{lemma}\label{lem:in>0}
With notation as above, for every $\bx\in \bR^d$, we have
\begin{align}\label{eq:Phi top >0}
\Phi_{f, H_X}^{\mathrm{top}} (\bx) \ge 0.
\end{align}
\end{lemma}
\begin{proof}
By continuity and homogeneity, it suffices to show 
\cref{eq:Phi top >0} for every $\bx\in \bZ^d$. 
Fix $\bx\in \bZ^d$.
Then by \cref{eq Phi mi >0}, for every integer $t\in \bZ$, we have
\[
\Phi_{f, H_X}(t\bx) = (f^{tx_1})^*H_X\cdots (f^{tx_d})^*H_X > 0.
\]
Note that by the definition of $\Phi_{f, H_X}^{\mathrm{top}}$, 
\[\Phi_{f, H_X}(t\bx)=\Phi_{f, H_X}^{\mathrm{top}} (\bx) \, t^{\deg \Phi_{f, H_X}}+O(t^{\deg \Phi_{f, H_X}-1}).\]
So \[\Phi_{f,H_X}^{\mathrm{top}}(\bx)=\lim_{t\to +\infty}\frac{\Phi_{f, H_X}(t\bx)}{t^{\deg \Phi_{f, H_X}}}\ge 0.\]
Thus \cref{lem:in>0} follows.
\end{proof}

\begin{lemma}\label{lem: top exists monomial all even}
Let $F(\bx)\in \bR[x_1,\ldots,x_d]$ be a nonzero polynomial that is nonnegative on $\bR^d$. Then $F$ contains, with nonzero coefficient, at least one monomial of square type, i.e., of the form $\bx^\lambda = x_1^{\lambda_1}x_2^{\lambda_2}\cdots x_d^{\lambda_d}$ for some $\lambda=(\lambda_1,\ldots, \lambda_d)\in (2\bN)^d$. 
\end{lemma}
\begin{proof}
Note that for any monomial $x_1^{\lambda_1}\cdots x_d^{\lambda_d}$,  if at least one $\lambda_i$ is odd, or equivalently $x_1^{\lambda_1}\cdots x_d^{\lambda_d}$ is not of square type, then
we have
\begin{align}
\sum_{\epsilon_1,\ldots, \epsilon_d\in\{-1,1\}}(\epsilon_1 x_1)^{\lambda_1}\cdots (\epsilon_d x_d)^{\lambda_d}=\Biggl(\prod_{i=1}^d(1+(-1)^{\lambda_i})\Biggr)x_1^{\lambda_1}\cdots x_d^{\lambda_d}=0.\label{eq:square type}
\end{align} 
Suppose to the contrary that all monomials with nonzero coefficients in $F$ are not of square type. Then by \cref{eq:square type}, we have
\begin{align*}
\sum_{\epsilon_1, \ldots, \epsilon_d\in\{-1,1\}}F(\epsilon_1 x_1, \ldots,  \epsilon_d x_d)=0,
\end{align*}
which implies that $F=0$ by nonnegativity, a contradiction. 
\end{proof}

We give a formula for the total degree of $\Phi_{f, H_X}$ in terms of non-vanishing of intersection numbers $w_\lambda$.
\begin{lemma}\label{lem:deg Phi}
With notation as above, we have
\begin{align}
\label{eq deg Phi=max}
\deg\Phi_{f, H_X} &=  \max\bigl\{ n : w_{\lambda}\neq 0 \, \text{ for some } \lambda\in P(k, d, n) \bigr\} \\
&=  \max\bigl\{ n : w_{\lambda}\neq 0 \, \text{ for some } \lambda\in P(k, d, n)\cap (2\bN)^d \bigr\}.\label{eq deg Phi=max even}
\end{align}
\end{lemma}

\begin{proof}
By \cref{eq:def Phi}, the coefficient of the monomial $\bx^\lambda$ in $\Phi_{f, H_X}$ is $w_\lambda/\lambda!$, and $w_\lambda$ is symmetric in the entries of $\lambda$, so \cref{eq deg Phi=max} follows.

By \cref{lem:in>0}, the polynomial $\Phi_{f,H_X}^{\mathrm{top}}$ is nonnegative on $\bR^d$.
Since it is also nonzero, \cref{lem: top exists monomial all even} implies that $\Phi_{f,H_X}^{\mathrm{top}}$ contains a monomial $\bx^\lambda$ with $\lambda \in (2\bN)^d$ and nonzero coefficient $w_\lambda/\lambda!$.
By the symmetry of $\Phi_{f, H_X}$, we may assume that $\lambda$ is a restricted partition.
It follows that $\deg\Phi_{f, H_X} = \deg\Phi_{f,H_X}^{\mathrm{top}} = |\lambda|$.
This proves \cref{eq deg Phi=max even}.
\end{proof}

The following elementary Riemann-sum asymptotic will be used to relate $\plov(f)$ to the degree of the dynamical intersection polynomial.

\begin{lemma}[Riemann-sum asymptotic for polynomials]
\label{lemma:Riemann-sum}
Let $F(\bx)\in \bR[x_1,\ldots,x_d]$ be a polynomial of total degree $D$, and let $F^{\mathrm{top}}(\bx)$ denote its homogeneous part of degree $D$.
Then
\[
\lim_{n\to\infty} \frac{1}{n^{d+D}} \sum_{0\le m_1,\ldots,m_d<n} F(m_1,\ldots,m_d) = \int_{[0,1]^d} F^{\mathrm{top}}(\bx)\, \diff \bx.
\]
Equivalently,
\[
\sum_{0\le m_1,\ldots,m_d<n} F(m_1,\ldots,m_d) = n^{d+D}\int_{[0,1]^d} F^{\mathrm{top}}(\bx)\, \diff \bx + O(n^{d+D-1}).
\]
\end{lemma}
\begin{proof}
Write $F(\bx) = F^{\mathrm{top}}(\bx) + F^{\le D-1}(\bx)$, where $F^{\le D-1}(\bx)$ has total degree at most $D-1$. Then there is a constant $C\geq 0$ such that
\[
|F^{\le D-1}(\bx)|\le C n^{D-1}
\qquad\text{for all } \bx \in [0,n]^d .
\]
Indeed, $C$ can be chosen as the sum of the absolute values of all coefficients of $F^{\le D-1}$, if any. 
Hence
\[
\sum_{0\le m_1,\ldots,m_d<n} F^{\le D-1}(m_1,\ldots,m_d)
=O(n^{d+D-1}).
\]
On the other hand, since $F^{\mathrm{top}}$ is homogeneous of degree $D$,
\[
\frac{1}{n^{d+D}}
\sum_{0\le m_1,\ldots,m_d<n} F^{\mathrm{top}}(m_1,\ldots,m_d)
=
\frac1{n^d}
\sum_{0\le m_1,\ldots,m_d<n}
F^{\mathrm{top}}\Bigl(\frac{m_1}{n},\ldots,\frac{m_d}{n}\Bigr),
\]
which converges to the Riemann integral
\[
\int_{[0,1]^d}F^{\mathrm{top}}(\bx)\,\diff \bx.
\]
This proves \cref{lemma:Riemann-sum}.
\end{proof}

The following theorem provides a new characterization of $\plov(f)$, i.e., the degree of $\Delta_n(f, H_X)^d$ viewed as a polynomial in $n$, using the dynamical intersection polynomial $\Phi_{f, H_X}$.
Our \cref{thm:main1} essentially follows from \cref{thm:plov=d+max,prop:thm3.2}.

\begin{theorem}
\label{thm:plov=d+max}
Let $X$ be a normal projective variety of dimension $d$ and let $H_X$ be an ample divisor on $X$. Let $f$ be a zero-entropy automorphism such that $f^*|_{\N^1(X)_\bR}$ is unipotent.
Let $\Phi_{f, H_X}$ denote the associated dynamical intersection polynomial (see \cref{def:Phi}).
Then we have
\begin{align}
\plov(f) &= d + \deg \Phi_{f, H_X} \label{eq:plov=d+deg Phi}\\
&= d + \max\bigl\{ n : w_{\lambda}\neq 0 \, \text{ for some } \lambda\in P(k, d, n) \bigr\} \label{eq:plov=d+max}\\
&= d + \max\bigl\{ n : w_{\lambda}\neq 0 \, \text{ for some } \lambda\in P(k, d, n) \cap (2\bN)^d\bigr\},\label{eq:plov=d+max even}
\end{align}
where $w_\lambda$ is the intersection number defined in \cref{eq:w-lambda}.
 Moreover, the leading coefficient of $\Delta_n(f, H_X)^d$ as a polynomial in $n$ is exactly 
\[
\int_{[0,1]^d} \Phi_{f, H_X}^{\mathrm{top}}(\bx) \, \diff \bx.
\]
\end{theorem}
\begin{proof}
For \cref{eq:plov=d+deg Phi,eq:plov=d+max,eq:plov=d+max even}, by \cref{lem:deg Phi}, it suffices to prove the first \cref{eq:plov=d+deg Phi}.
Denote the total degree $\deg \Phi_{f, H_X}$ of $\Phi_{f, H_X}$ by $d_\Phi$ for simplicity.
Recall that by \cite[Lemma~2.16]{LOZ25} (see also \cref{lemma:plov}), $\plov(f)$ is equal to the degree of the polynomial $\Delta_n(f, H_X)^d$ in $n$. 
By the definitions of $\Delta_n(f, H_X)$ and dynamical intersection polynomial $\Phi_{f, H_X}$, we have
\begin{align*}
\Delta_{n}(f, H_X)^d &= \left(\sum_{m=0}^{n-1} (f^m)^*H_X \right)^{\! d}\\
&= \sum_{0\leq m_1, \ldots, m_d<n}(f^{m_1})^*H_X\cdots (f^{m_d})^*H_X\\
&= \sum_{0\leq m_1, \ldots, m_d<n}\Phi_{f, H_X}(m_1, \ldots, m_d)\\
&= n^{d+d_\Phi}\int_{[0,1]^d} \Phi_{f, H_X}^{\mathrm{top}}(\bx) \, \diff \bx+O(n^{d+d_\Phi-1}).
\end{align*}
For the last equality, we apply \cref{lemma:Riemann-sum} to $F=\Phi_{f,H_X}$.
Here note that $\Phi_{f, H_X}^{\mathrm{top}}(\bx)$ is not identically $0$ by definition, hence it is almost everywhere positive on $[0, 1]^d$ by \cref{lem:in>0}, which implies that $ \int_{[0,1]^d} \Phi_{f, H_X}^{\mathrm{top}}(\bx) \, \diff \bx>0$. This proves that 
\[
\plov(f)=\deg_n\Delta_{n}(f, H_X)^d=d+d_\Phi=d+\deg \Phi_{f, H_X}
\]
and the leading coefficient of $\Delta_n(f, H_X)^d$ is exactly 
$
\int_{[0,1]^d} \Phi_{f, H_X}^{\mathrm{top}}(\bx) \, \diff \bx.$ 
\end{proof}

\begin{remark}
\label{rmk:cancellation}
Previously, because of the potential cancellation, it was only known that  
\[
\plov(f) \le d + \max\bigl\{ n : w_{\lambda}\neq 0 \text{ for some } \lambda\in P(k, d, n) \bigr\};
\]
see \cite[Lemma~2.16]{LOZ25} or \cite[Lemma~2.1]{HJ25}. 
By proving the \cref{eq:plov=d+max}, \cref{thm:plov=d+max} affirmatively answers an open question of the authors \cite[Question~2.2]{HJ25}.
On the other hand, it is also easy to see from \cref{eq:plov=d+max even} that $\plov(f)-d$ is divisible by $2$, yielding another proof of the parity principle, i.e., \cref{prop:same-parity}.
\end{remark}

\subsection{Vanishing and non-vanishing of intersection numbers}
\label{subsec:intersection-numbers}

As already observed in \cite{LOZ25, HJ25}, in order to apply \cref{thm:plov=d+max}, it is essential to study the vanishing or non-vanishing property for the intersection numbers $w_\lambda$.

With the help of the new operator $L$, the following lemma provides the desired homogeneous linear equations on these intersection numbers $w_\lambda$. 

\begin{lemma}
\label{lemma:equations-v-lambda}
For each integer $n$ with $1\leq n\leq dk$, the intersection numbers $w_\lambda$, for $\lambda \in P(k,d,n)$, defined in \cref{eq:w-lambda}, satisfy the following system of homogeneous linear equations:
\[
A_{k,d,n} \cdot (w_\lambda)_{\lambda\in P(k,d,n)} = \mathbf{0},
\]
where $A_{k,d,n}\in \bR^{p(k,d,n-1)\times p(k,d,n)}$ is the weighted incidence matrix defined in \cref{def:Akdn}.
\end{lemma}
\begin{proof}
Using the fact that $f^*$ preserves the intersection form, we can show that
\begin{equation}
\label{eq:derivative}
\sum_{j=1}^d D_1 \cdots D_{j-1} \cdot (LD_j) \cdot D_{j+1} \cdots D_d = 0,
\end{equation}
for any divisors $D_1,\ldots,D_d$ in $\N^1(X)_\bR$.
Indeed, consider the real-valued function
\[
\varphi(t) \coloneqq \prod_{j=1}^{d} \bigl(\exp(tL) D_j\bigr) = \prod_{j=1}^{d} \, \sum_{i=0}^{k}\frac{t^i}{i!} L^i D_j = \sum_{0\leq \lambda_1,\ldots,\lambda_d \leq k} \!\!\! \frac{(L^{\lambda_1}D_1) \cdots (L^{\lambda_d}D_d)}{\lambda_1! \cdots \lambda_d!} \, t^{\lambda_1+\cdots+\lambda_d}.
\]
This is a polynomial in $t$ of degree at most $dk$.
Since $f^*|_{\N^1(X)_\bR} = \exp(L)$ preserves the intersection form, $\varphi(t)$ is the constant $D_1\cdots D_d$ for every $t\in \bZ$ and hence for every $t\in \bR$.
It follows that the coefficient of the linear term in $\varphi(t)$ vanishes.
This proves \cref{eq:derivative}.

To obtain the desired system of homogeneous linear equations, it suffices to let $D_1,\ldots,D_d$ vary among $H_X, LH_X,\ldots, L^kH_X$ in \cref{eq:derivative}. 
Precisely, for any $n\geq 1$ and an arbitrary $\mu = (\mu_1,\mu_2,\ldots,\mu_d) = (k^{e_k},\ldots,1^{e_1},0^{e_0}) \in P(k,d,n-1)$,
choosing $D_j=L^{\mu_j}H_X$ for all $1\leq j\leq d$ in \cref{eq:derivative}, 
we obtain the following equation:
\[
\sum_{\substack{0\leq i\leq k-1 \\[1pt] e_i>0}} e_i \, (L^{k}H_X)^{e_k} \cdots (L^{i+1}H_X)^{e_{i+1}+1} \cdot (L^{i}H_X)^{e_{i}-1} \cdots H_X^{e_0} = \sum_{\substack{0\leq i\leq k-1 \\[1pt] e_i>0}} e_i \, w_{\mu(i)} = 0,
\]
where $\mu(i)$ is the same as in the definition of $A_{k,d,n}$; see \cref{def:Akdn}.
This is exactly the $\mu$-th equation in the system $A_{k,d,n} \cdot (w_\lambda)_{\lambda\in P(k,d,n)} = \mathbf{0}$.
\end{proof}

We note that the transition matrix from the divisors $\{N^iH_X\}_{0\le i\le k}$ to the divisors $\{L^iH_X\}_{0\le i\le k}$ in $\N^1(X)_\bR$ is upper-triangular and unipotent.
This allows us to compare the intersection numbers defined using $N$ with those defined using $L$.
We first provide the following combinatorial vanishing property on intersection numbers, which is an analogue of \cite[Claim~5.1]{HJ25} after replacing $N$ by $L$.

\begin{proposition}[{Combinatorial vanishing, cf.~\cite[Claim~5.1]{HJ25}}]
\label{prop:vanishing-after-kd/2}
Assume the setting and the notation of \cref{assum:XfN}.
For any restricted partition $\lambda = (k^{e_k},\ldots,1^{e_1},0^{e_0}) \in P(k,d,n)$, if $n>dk/2$, then
\[
w_\lambda \coloneqq (L^{k}H_X)^{e_k} \cdots (LH_X)^{e_1} \cdot H_X^{e_0} = 0.
\]
\end{proposition}
\begin{proof}
It follows from \cref{lemma:equations-v-lambda} and \cref{thm:full-rank}.
\end{proof}

Besides \cref{prop:vanishing-after-kd/2}, we also have the following vanishing and positivity properties, that essentially follow from \cite[Theorem~3.2]{HJ25} with $N$ replaced by $L$.
We include the proof for the reader's convenience.
It turns out to be helpful to reformulate the vanishing statement \cite[Theorem~3.2(4)]{HJ25} in terms of restricted partitions $\lambda\in P(k,d,n)$ satisfying $\lambda \succ \kappa(\bt)$, as in \cref{prop:thm3.2}\ref{prop:thm3.2-1}.
This motivates the notion of truncation introduced in \Cref{subsec:truncated-matrices}, which will play a crucial role in the proof of \cref{thm:k2d-4} in \Cref{sec:proof}.

\begin{proposition}[{cf.~\cite[Theorem~3.2]{HJ25}}]
\label{prop:thm3.2}
Assume the setting and the notation of \cref{assum:XfN}.
Then there exists an $(r+1)$-tuple $\bt \coloneqq (t_0,t_1,\ldots,t_r)\in \bZ_{>0}^{r+1}$ with $\sum_{j=0}^r t_j=d$ and an associated restricted partition
\[
\kappa(\bt) \coloneqq \bigl((2r)^{t_r},(2r-2)^{t_{r-1}},\ldots,2^{t_1},0^{t_0}\bigr) \in P\Bigl(2r,d,\sum_{j=0}^r 2jt_j\Bigr)
\]
satisfying the following properties:
\begin{enumerate}[label=\emph{(\arabic*)}, ref=(\arabic*)]
\item \label{prop:thm3.2-1} \emph{Geometric vanishing:} For any \[\lambda = ((2r)^{e_{2r}},(2r-1)^{e_{2r-1}},\ldots,1^{e_1},0^{e_0}) \in P\Bigl(2r,d,\sum_{i=0}^{2r}ie_i\Bigr)\] with $\lambda\succ \kappa(\bt)$ in the lexicographic order, one has
\[
w_\lambda = (L^{2r}H_X)^{e_{2r}} \cdot (L^{2r-1}H_X)^{e_{2r-1}} \cdots (LH_X)^{e_1} \cdot H_X^{e_0} = 0.
\]

\item \label{prop:thm3.2-2} \emph{Geometric positivity:} One has $w_{\kappa(\bt)} = (L^{2r}H_X)^{t_{r}} \cdot (L^{2r-2}H_X)^{t_{r-1}} \cdots H_X^{t_0} > 0$.

\item \label{prop:thm3.2-3} \emph{Boundedness:} The integers $t_j$ satisfy that
\begin{equation*}
r(r+1) \le \sum_{j=0}^r 2j  t_j \le rd.
\end{equation*}
\end{enumerate}
\end{proposition}

\begin{proof}
Let $t_1,\ldots,t_r$ be the positive integers given by \cite[Theorem~3.2]{HJ25}, and set $t_0 \coloneqq d-\sum_{j=1}^{r} t_j \in \bZ_{>0}$.
Define $M_0^{(L)} \coloneqq [X] \eqqcolon M_0^{(N)}$, and for each $1\le j\le r$, let
\begin{align*}
M_j^{(L)} &\coloneqq (L^{2r}H_X)^{t_r} \cdot (L^{2r-2}H_X)^{t_{r-1}}\cdots (L^{2r-2j+2}H_X)^{t_{r-j+1}},\\
M_j^{(N)} &\coloneqq (N^{2r}H_X)^{t_r} \cdot (N^{2r-2}H_X)^{t_{r-1}}\cdots (N^{2r-2j+2}H_X)^{t_{r-j+1}}.
\end{align*}
Here $M_j^{(N)}$ is denoted by $M_j$ in \cite[Theorem~3.2]{HJ25}.
Recall that \[L \coloneqq \log(f^*|_{\N^1(X)_\bR}) = \sum_{i=1}^{k} \frac{(-1)^{i+1}}{i} N^i.\]
So, for each $m$ with $1\le m\le 2r$, we have
\[
L^m H_X \num N^m H_X + \sum_{n\ge m+1} c_{m,n} \, N^n H_X,
\]
for suitable constants $c_{m,n}\in\bQ$.
This leads to the following comparison.

\begin{claim}
\label{claim:comparison}
For every $0\le j\le r$, one has
\begin{equation*}
M_j^{(L)}\cdot (L^{s}H_X)^i \wnum M_j^{(N)}\cdot (N^{s}H_X)^i, \quad \text{for all } s\geq 2r-2j \text{ and } i\ge 0,
\end{equation*}
where we refer to \cite[Definition~2.4]{HJ25} for the notion of weak numerical equivalence $\wnum$.
\end{claim}
\begin{proof}[Proof of \cref{claim:comparison}]
We prove this claim by induction on $j$.
The case when $j=0$ is clear, since $L^{2r}H_X \num N^{2r}H_X$ and $L^{s}H_X \num N^{s}H_X \num 0$ for all $s>2r$.
Assume now that \cref{claim:comparison} holds for some $j$ with $0\le j\le r-1$.
Then by taking $s=2r-2j$ and $i=t_{r-j}$, we readily have that
\[
M_{j+1}^{(L)} = M_j^{(L)}\cdot (L^{2r-2j}H_X)^{t_{r-j}} \wnum M_j^{(N)}\cdot (N^{2r-2j}H_X)^{t_{r-j}}= M_{j+1}^{(N)}.
\]

Now let $s\geq 2r-2j-2$ and $i>0$. 
We have
\begin{align*}
M_{j+1}^{(L)}\cdot (L^{s}H_X)^i
&\wnum M_{j+1}^{(N)}\cdot (L^{s}H_X)^i
\wnum M_{j+1}^{(N)}\cdot \Biggl(N^{s}H_X + \sum_{n\ge s+1} c_{s,n} \, N^n H_X \Biggr)^{\! i}.
\end{align*}
After expanding the right-hand side, every term except $M_{j+1}^{(N)}\cdot (N^{s}H_X)^{i}$ contains at least one factor $M_{j+1}^{(N)} \cdot (N^nH_X)$ with $n\ge s+1\geq  2r-2j-1$. By \cite[Theorem~3.2(4)]{HJ25}, every such term is weakly numerically trivial. Hence we have 
\[
M_{j+1}^{(L)}\cdot (L^{s}H_X)^i \wnum M_{j+1}^{(N)}\cdot (N^{s}H_X)^i.
\]
This completes the proof of the comparison \cref{claim:comparison} by induction.
\renewcommand{\qedsymbol}{}
\end{proof}

Now the geometric positivity assertion~\ref{prop:thm3.2-2} for $w_{\kappa(\bt)}$ readily follows from the corresponding assertion for $N$ in \cite[Theorem~3.2(6)]{HJ25} or \cite[Theorem~1.5]{HJ25}, together with the comparison \cref{claim:comparison}.
Combining this with \cref{prop:vanishing-after-kd/2} immediately yields the boundedness assertion~\ref{prop:thm3.2-3}.

It remains to show the vanishing assertion~\ref{prop:thm3.2-1}.
First, by \cite[Theorem~3.2(4)]{HJ25} and \cref{claim:comparison}, for each $0\le j\le r$, we have
\begin{align}\label{eq:MLsH=0}
M_{j}^{(L)}\cdot L^s H_X \wnum M_{j}^{(N)}\cdot N^s H_X \wnum 0 \quad\text{for all } s \ge 2r-2j+1.
\end{align}
Now fix $\lambda = ((2r)^{e_{2r}},(2r-1)^{e_{2r-1}},\ldots,1^{e_1},0^{e_0}) \in P(2r,d,n)$ with $\lambda \succ \kappa(\bt)$.
Let $m$ be the first index such that
$\lambda_m > (\kappa(\bt))_m$.
Then by definition, $(\kappa(\bt))_m=2\ell$ for some $0\le \ell\le r-1$. Hence the first $m-1$ parts of $\lambda$ are precisely
\[
(2r)^{t_r},(2r-2)^{t_{r-1}},\ldots,(2\ell+2)^{t_{\ell+1}},
\]
i.e., $e_{2i}=t_i$ for all $r\ge i\ge \ell+1$ and $e_{2i+1}=0$ for all $r-1 \ge i\ge \ell+1$.
Moreover, we necessarily have that $\lambda_m = 2\ell+1$ or $2\ell+2$.
Taking $j=r-\ell$, we have
\[
M_{j}^{(L)}
=(L^{2r}H_X)^{t_r} \cdot (L^{2r-2}H_X)^{t_{r-1}}\cdots (L^{2\ell+2}H_X)^{t_{\ell+1}},
\]
and $2r-2j+1=2\ell+1$.
It thus follows from \cref{eq:MLsH=0} that
\[
M_{j}^{(L)}\cdot L^{\lambda_m}H_X\wnum 0.
\]
Multiplying by the remaining divisor factors appearing in $w_\lambda$, we obtain that $w_\lambda=0$.
Thus $\lambda\succ \kappa(\bt)$ implies that $w_\lambda=0$, proving \ref{prop:thm3.2-1}.
We conclude the proof of \cref{prop:thm3.2}.
\end{proof}

\begin{remark}\label{rmk:unique-bt}
In \cref{prop:thm3.2}, it is easy to see that the $(r+1)$-tuple 
$\bt\in \bZ_{>0}^{r+1}$ satisfying both properties~\ref{prop:thm3.2-1} and \ref{prop:thm3.2-2} is unique. 
Indeed, properties~\ref{prop:thm3.2-1} and \ref{prop:thm3.2-2} guarantee that $\kappa(\bt)$ is exactly the lexicographically largest element in the set
\[
\bigl\{\kappa(\bt'): w_{\kappa(\bt')}\neq 0, \, \bt'\in \bZ_{>0}^{r+1}, \, |\bt'|=d \bigr\}.
\]
Although we will not use this here, \cref{claim:comparison} shows that all assertions of \cite[Theorem~3.2]{HJ25} remain valid verbatim after replacing $N$ by $L$.
\end{remark}

\section{Truncated weighted incidence matrices}
\label{sec:truncated}

When studying intersection numbers associated with partitions, the vanishing property in \cref{prop:thm3.2}\ref{prop:thm3.2-1} will allow us to disregard those partitions that are lexicographically larger than a certain distinguished partition $\kappa$. 
Accordingly, in \Cref{subsec:truncated-matrices}, we introduce truncated partition sets and truncated matrices with respect to such distinguished partitions.

In \Cref{subsec:truncated-nullity}, we prove the main result of this section; see \cref{thm:nullity-truncated-matrix2}.
It establishes, by induction, that the truncated matrix $A_{2d-4,d,(d-1)(d-2)+e}^{\preceq \kappa}$ has trivial nullspace for each $e$ that is not too small.
This is a key ingredient for obtaining the desired upper bound in \cref{thm:k2d-4}.

\subsection{Truncated matrices with respect to distinguished partitions}
\label{subsec:truncated-matrices}

\begin{definition}[Truncated matrices $A_{2r,d,n}^{\preceq \kappa(\bt)}$]
\label{def:truncated-matrix}
Fix positive integers $d$, $r$ with $1\le r\le d-1$, and $n$ with $1\le n \le 2rd$. 
For a vector $\bt = (t_0,t_1,\ldots,t_r)\in \bZ_{>0}^{r+1}$ with $\sum_{j=0}^r t_j=d$,
denote by 
\[
\kappa(\bt) \coloneqq \bigl((2r)^{t_r},(2r-2)^{t_{r-1}},\ldots,2^{t_1},0^{t_0}\bigr) \in P\Bigl(2r,d,\sum_{j=0}^r 2jt_j\Bigr),
\]
which is a distinguished restricted partition associated with $\bt$.
We consider the following truncated set of restricted partitions:
\[
P(2r, d, n)^{\preceq \kappa(\bt)} \coloneqq \bigl\{ \lambda\in P(2r, d, n) \,:\, \lambda\preceq \kappa(\bt)\bigr\},
\]
whose cardinality is denoted by $p(2r, d, n)^{\preceq \kappa(\bt)}$.

Next, we define $A_{2r,d,n}^{\preceq \kappa(\bt)}$ to be the submatrix of $A_{2r,d,n}$ (see \cref{def:Akdn}) indexed by
\[
P(2r, d, n-1)\times P(2r, d, n)^{\preceq \kappa(\bt)},
\]
and call it the \textit{truncated matrix} of $A_{2r,d,n}$ with respect to $\kappa(\bt)$.
In other words, $A_{2r,d,n}^{\preceq \kappa(\bt)}$ is the matrix obtained from $A_{2r,d,n}$ by deleting all columns indexed by restricted partitions that are lexicographically larger than $\kappa(\bt)$. 
\end{definition}

We first prepare some general results on truncated matrices. We will use them in \Cref{subsec:truncated-nullity} to determine the nullity of certain truncated matrices that arise naturally, from an intersection-theoretic viewpoint, in the proof of \cref{thm:k2d-4}.

For $\bt = (t_0,t_1,\ldots,t_r)\in \bZ_{>0}^{r+1}$, define
\[
\bt^-=
\begin{cases}
    (t_0,t_1,\ldots,t_{r-1})\in \bZ_{>0}^{r}, & \text{if } t_r=1;\\
     (t_0,t_1,\ldots,t_{r-1}, t_r-1)\in \bZ_{>0}^{r+1}, & \text{if } t_r>1.
\end{cases}
\]
Accordingly, the corresponding restricted partitions satisfy that $\kappa(\bt) = \iota_{2r}(\kappa(\bt^-))$.

\begin{lemma}
\label{lemma:P2rd-trun-intersection}
Fix $\bt = (t_0,t_1,\ldots,t_r)\in \bZ_{>0}^{r+1}$ with $\sum_{i=0}^r t_i =d$.
Then the truncated set $P(2r, d, n)^{\preceq \kappa(\bt)}$ satisfies the following intersection properties:
\begin{enumerate}[label=\emph{(\arabic*)}, ref=(\arabic*)]
\item \label{lemma:P2rd-trun-intersection-1} For any $k'<2r$, one has
\[
P(2r, d, n)^{\preceq \kappa(\bt)} \cap P(k', d, n) = P(k', d, n).
\]

\item \label{lemma:P2rd-trun-intersection-2} \vspace{5pt}
\hfil $\begin{aligned}[t]
&P(2r, d, n)^{\preceq \kappa(\bt)} \cap \iota_{2r}(P(2r, d-1, n-2r)) \\
={}& \begin{cases}
\iota_{2r}(P(2r-2, d-1, n-2r)^{\preceq \kappa(\bt^-)}), &\text{if } t_r=1;\\
\iota_{2r}(P(2r, d-1, n-2r)^{\preceq \kappa(\bt^-)}), &\text{if } t_r>1.
\end{cases}
\end{aligned}$
\end{enumerate}
\end{lemma}

\begin{proof}
For \ref{lemma:P2rd-trun-intersection-1}, it suffices to show that $P(k', d, n)\subseteq P(2r, d, n)^{\preceq \kappa(\bt)}$ for $k'<2r$. Clearly, $P(k', d, n)\subseteq P(2r, d, n)$. Moreover, for any $\lambda \in P(k', d, n)$, the largest part of $\lambda$ is at most $k'<2r$, whereas the largest part of $\kappa(\bt)$ is $2r$.
Hence $\lambda \prec \kappa(\bt)$, and therefore $\lambda \in P(2r, d, n)^{\preceq \kappa(\bt)}$.
This proves \ref{lemma:P2rd-trun-intersection-1}.

For \ref{lemma:P2rd-trun-intersection-2}, 
take any $\lambda=(\lambda_1,\ldots, \lambda_{d-1})\in P(2r, d-1, n-2r)$.
Then $\iota_{2r}(\lambda)=(2r, \lambda)=(2r, \lambda_1,\ldots, \lambda_{d-1})$.
Since $\kappa(\bt)=(2r, \kappa(\bt^-)) = \iota_{2r}(\kappa(\bt^-))$, we have that
\[
\iota_{2r}(\lambda)\preceq \kappa(\bt) \iff \lambda \preceq \kappa(\bt^-).
\]
If $t_r=1$, then the largest part of $\kappa(\bt^-)$ is at most $2r-2$.
So $\lambda \preceq \kappa(\bt^-)$ implies that $\lambda_1\leq 2r-2$, i.e., $\lambda\in P(2r-2,d-1,n-2r)$.
This proves the first case of \ref{lemma:P2rd-trun-intersection-2}. The second case is analogous.
\end{proof}

\begin{lemma}
\label{lemma:P2rd-trun-decomposition}
Fix $\bt = (t_0,t_1,\ldots,t_r)\in \bZ_{>0}^{r+1}$ with $\sum_{i=0}^r t_i =d$.
Then the truncated set $P(2r, d, n)^{\preceq \kappa(\bt)}$ admits the following decomposition:
\[
P(2r, d, n)^{\preceq \kappa(\bt)} =\begin{cases}
    \iota_{2r}(P(2r-2, d-1, n-2r)^{\preceq \kappa(\bt^-)})\sqcup P(2r-1, d, n), & \text{if } t_r=1;\\
      \iota_{2r}(P(2r, d-1, n-2r)^{\preceq \kappa(\bt^-)})\sqcup P(2r-1, d, n), & \text{if } t_r>1.
\end{cases} 
\]
\end{lemma}
\begin{proof}
Recall that the decomposition \eqref{eq:decomposition-by-iota} reads as follows:
\begin{align*}
P(2r, d, n)= \iota_{2r}(P(2r, d-1, n-2r))\sqcup P(2r-1, d, n).
\end{align*}
Intersecting both sides with $P(2r, d, n)^{\preceq \kappa(\bt)}$, we get the decomposition by \cref{lemma:P2rd-trun-intersection}.
\end{proof}

Under the natural decompositions of $P(2r, d, n-1)$ and $P(2r, d, n)^{\preceq \kappa(\bt)}$, the truncated matrix $A_{2r,d,n}^{\preceq \kappa(\bt)}$ has the following block lower-triangular form, inherited from that of $A_{2r,d,n}$; see \cref{lemma:Akdn-decomposition}.

\begin{lemma}
\label{lemma:Akdn-trun-decomposition}
Fix $\bt = (t_0,t_1,\ldots,t_r)\in \bZ_{>0}^{r+1}$ with $\sum_{i=0}^r t_i =d$. Then the truncated matrix $A_{2r,d,n}^{\preceq \kappa(\bt)}$ has the following block lower-triangular form:
\begin{enumerate}[label=\emph{(\arabic*)}, ref=(\arabic*)]
\item If $t_r=1$, then
\[
A_{2r,d,n}^{\preceq \kappa(\bt)}
=
\begin{pmatrix} \mathbf{0}& \mathbf{0}\\
A_{2r-2, d-1, n-2r}^{\preceq \kappa(\bt^-)} & \mathbf{0}\\
\ast & A_{2r-1, d, n}
\end{pmatrix}.
\]
 
\item If $t_r>1$, then
\[
A_{2r,d,n}^{\preceq \kappa(\bt)}
=
\begin{pmatrix} \mathbf{0}& \mathbf{0}\\
A_{2r, d-1, n-2r}^{\preceq \kappa(\bt^-)} & \mathbf{0}\\
\ast & A_{2r-1, d, n}
\end{pmatrix}.
\]
\end{enumerate}

\end{lemma}
\begin{proof}
We prove only the case when $t_r=1$, since the other case is analogous.
By \cref{lemma:Akdn-decomposition}, the matrix $A_{2r,d,n}$ has the following block lower-triangular form
\[
A_{2r,d,n} =
\begin{pmatrix}
A_{2r,d-1,n-2r} & \mathbf{0}\\
\ast & A_{2r-1,d,n}
\end{pmatrix}.
\]
Recall that the rows of both $A_{2r,d,n}$ and $A_{2r,d,n}^{\preceq \kappa(\bt)}$ are indexed by
\[
P(2r, d, n-1) = \iota_{2r}(P(2r, d-1, n-2r-1))\sqcup P(2r-1, d, n-1).
\]
Further, by \cref{lemma:P2rd-trun-decomposition}, the columns of $A_{2r,d,n}^{\preceq \kappa(\bt)}$ are indexed by
\[
P(2r, d, n)^{\preceq \kappa(\bt)} = \iota_{2r}(P(2r-2, d-1, n-2r)^{\preceq \kappa(\bt^-)})\sqcup P(2r-1, d, n).
\]
It is clear that the bottom-right block of $A_{2r,d,n}$ remains unchanged in $A_{2r,d,n}^{\preceq \kappa(\bt)}$, where it is naturally identified with $A_{2r-1,d,n}$.

Now let $\mu\in P(2r, d-1, n-2r-1)$ and $\lambda\in P(2r-2, d-1, n-2r)^{\preceq \kappa(\bt^-)}$ be arbitrary restricted partitions.
Then
\[
(2r,\mu) \in \iota_{2r}(P(2r, d-1, n-2r-1)) \ \  \text{and} \ \  (2r,\lambda)\in \iota_{2r}(P(2r-2, d-1, n-2r)^{\preceq \kappa(\bt^-)})
\]
index a row and a column of $A_{2r,d,n}^{\preceq \kappa(\bt)}$, respectively.
If the largest part of $\mu$ exceeds $2r-2$, then no partition of the form $(2r, \lambda)$ with $\lambda \preceq \kappa(\bt^-)$ can be obtained from $(2r, \mu)$ by increasing a single part by $1$, since the largest part of $\lambda$ is at most $2r-2$. The corresponding entries of $A_{2r,d,n}^{\preceq \kappa(\bt)}$ therefore vanish, which gives the top-left zero block.
If instead the largest part of $\mu$ is at most $2r-2$, then $\mu \in P(2r-2, d-1, n-2r-1)$, and the $((2r,\mu), (2r,\lambda))$-entry of $A_{2r,d,n}^{\preceq \kappa(\bt)}$ coincides with the $(\mu, \lambda)$-entry of $A_{2r-2, d-1, n-2r}^{\preceq \kappa(\bt^-)}$. This yields the middle-left block.
\end{proof}

For an arbitrary matrix $M$, the nullity of $M$ is the dimension of the kernel of the linear map represented by $M$, which equals the number of columns of $M$ minus the rank of $M$.

\begin{lemma}
\label{lemma:nullity-lower-triangular}
Let
\[
M=\begin{pNiceArray}{cc}[first-row,last-col]
\Hbrace{1}{n_1} & \Hbrace{1}{n_2} \\
A &  \mathbf{0}  & \Vbrace{1}{m_1} \\[2pt]
B & C & \Vbrace{1}{m_2} \\
\end{pNiceArray}
\]
be a block lower-triangular matrix over $\bR$. Here possibly $m_1=0$ or $n_2=0$.
Then
\begin{equation}
\label{eq:nullity-formula}
\nullity(M) = \dim_\bR \bigl\{x\in \nullspace(A) \,:\, Bx\in \range(C) \bigr\}+\nullity(C).
\end{equation}
In particular, the following hold.
\begin{enumerate}[label=\emph{(\arabic*)}, ref=(\arabic*)]
\item If $C$ has full row rank, then 
\[
\nullity(M)=\nullity(A)+\nullity(C).
\]

\item If $C$ has full column rank, then 
\[
\nullity(M) \leq \nullity(A).
\]
\end{enumerate}
\end{lemma}

\begin{proof}
Let $\binom{x}{y}$ be an arbitrary column vector in the domain of $M$, with $x\in \bR^{n_1\times 1}$ and $y\in \bR^{n_2\times 1}$.
Then
\[
M \binom{x}{y} = \binom{Ax}{Bx+Cy}.
\]
Thus $\binom{x}{y}\in \nullspace(M)$ if and only if $Ax=\mathbf{0}$ and $Cy=-Bx$.
For a given $x\in \nullspace(A)$, the second condition is solvable in $y$ if and only if $-Bx\in \range(C)$, or equivalently, $Bx\in \range(C)$.
Moreover, the solutions $y$ of $Cy=-Bx$ form an affine translate of
$\nullspace(C)$, so they contribute $\nullity(C)$ free parameters.
The nullity formula \cref{eq:nullity-formula} follows immediately.

If $C$ has full row rank, then $\range(C)=\bR^{m_2\times 1}$, so
\[
\nullity(M)=\nullity(A)+\nullity(C).
\]
 
If $C$ has full column rank, then $\nullity(C)=0$, and
\[
\nullity(M)=\dim_\bR \bigl\{x\in \nullspace(A) \,:\, Bx\in \range(C) \bigr\} 
\leq \nullity(A). \qedhere
\]
\end{proof}

\subsection{The nullity of truncated matrices}
\label{subsec:truncated-nullity}

We prove a technical theorem on the nullity of truncated matrices $A_{2d-4,d,(d-1)(d-2)+e}$.
In the Euclidean space $\bR^{d-1}$, let $\be_i$ denote the $i$-th standard basis vector for $1\le i\le d-1$ and $\mathbf{1}_{d-1}=(1,1,\ldots, 1)$.

\begin{theorem}\label{thm:nullity-truncated-matrix2}
For each $d\ge 3$ and $0\le \ell\le d-2$, 
denote by $\bt_{d-2, \ell}=\mathbf{1}_{d-1}+\be_{\ell+1}$ so that
\[
\kappa(\bt_{d-2, \ell})=(2d-4,2d-6,\ldots,2\ell,2\ell,\ldots,2,0) \in P(2d-4,d,(d-1)(d-2)+2\ell).
\]
Then, for every integer
\[
e\ge \max\bigl\{\,\lfloor d/2\rfloor-1, \, \ell+1 \,\bigr\},
\]
we have
\[
\nullity\bigl(A_{2d-4,d,(d-1)(d-2)+e}^{\preceq \kappa(\bt_{d-2, \ell})}\bigr)=0.
\] 
\end{theorem}

\begin{proof}
Write $r=d-2$ and fix an integer $\ell$ with $0\le \ell\le r$.
Set
\[
\kappa_{r,\ell}\coloneqq \kappa(\bt_{r,\ell})
= (2r,2r-2,\ldots,2\ell,2\ell,\ldots,2,0).
\]
Fix an integer $e$ such that $e\ge \max\bigl\{\, \lfloor r/2\rfloor, \, \ell+1 \,\bigr\}$.

We prove the theorem by induction on $r-\ell$.

If $\ell=r$, then the assumption $e\ge \ell+1$ becomes $e\ge r+1$.
Hence
\[
r(r+1)+e > \frac{2r(r+2)}{2}=r(r+2),
\]
so by \cref{thm:full-rank}, the matrix $A_{2r,r+2,r(r+1)+e}$ has full column rank.
Therefore, its submatrix $A_{2r,r+2,r(r+1)+e}^{\preceq \kappa_{r, \ell}}$, which has the same row set, also has full column rank. This proves the base case.

Now suppose that $0\le \ell<r$ and assume, by induction, that the statement has already been proved for $(r-1,\ell)$.
Then $\bt_{r, \ell}^-=\bt_{r-1, \ell}$, and by \cref{lemma:Akdn-trun-decomposition}, the truncated matrix $A_{2r,r+2,r(r+1)+e}^{\preceq \kappa_{r,\ell}}$ has the block form
\[
A_{2r,r+2,r(r+1)+e}^{\preceq \kappa_{r,\ell}}
=
\begin{pmatrix}
\mathbf{0} & \mathbf{0} \\
A_{2r-2,r+1,r(r-1)+e}^{\preceq \kappa_{r-1,\ell}} & \mathbf{0} \\
\ast & A_{2r-1,r+2,r(r+1)+e}
\end{pmatrix}.
\]
Since $e\ge \lfloor r/2\rfloor$, we have
\[
r(r+1)+e > \frac{(2r-1)(r+2)}{2},
\]
so \cref{thm:full-rank} implies that the bottom-right block $A_{2r-1,r+2,r(r+1)+e}$ has full column rank.
It thus follows from \cref{lemma:nullity-lower-triangular} that
\[
\nullity\bigl(A_{2r,r+2,r(r+1)+e}^{\preceq \kappa_{r,\ell}} \bigr) \le \nullity\bigl(A_{2r-2,r+1,r(r-1)+e}^{\preceq \kappa_{r-1,\ell}} \bigr).
\]
On the other hand, from the assumption on $e$ we have $e\ge \max \bigl\{\, \lfloor (r-1)/2\rfloor, \, \ell+1 \,\bigr\}$.
Therefore, $\nullity\bigl(A_{2r-2,r+1,r(r-1)+e}^{\preceq \kappa_{r-1,\ell}} \bigr)=0$ by the induction hypothesis, and hence
\[
\nullity\bigl(A_{2r,r+2,r(r+1)+e}^{\preceq \kappa_{r,\ell}} \bigr)=0.
\]
This completes the induction.
\end{proof}

\begin{remark}
\label{rmk:trun-A-table}
The lower bound on $e$ in \cref{thm:nullity-truncated-matrix2} is chosen so that our inductive argument works.
To illustrate, Table~\ref{table:nullity} below lists the nullity of truncated matrices $A_{2d-4,d,(d-1)(d-2)+e}^{\preceq \kappa(\bt_{d-2, 0})}$ with respect to the restricted partition $\kappa(\bt_{d-2, 0})=(2d-4,2d-6,\ldots,2,0,0)$, for $3\le d\le 9$ and $e\ge 0$.
Note that when $e$ is close to zero, the conclusion of \cref{thm:nullity-truncated-matrix2} is no longer true.

\begin{table}[htbp]
\setlength{\tabcolsep}{1.5em}
\renewcommand{\arraystretch}{1.2}
\caption{Nullity of the truncated matrix $A_{2d-4,d,(d-1)(d-2)+e}^{\preceq \kappa(\bt_{d-2, 0})}$}
\label{table:nullity}
\begin{tabular}{|c|c|c|c|c|c|c|}
\hline
\diagbox[width=6em]{$d$}{$e$} & $0$ & $1$ & $2$ & $3$ & $4$ & $5$ \\
\hline
$3$ & $1$ & $0$ & $0$ & $0$ & $0$ &  \\\hline
$4$ & $2$ & $0$ & $0$ & $0$ & $0$ & $0$ \\
\hline
$5$ & $3$ & $0$ & $0$ & $0$ & $0$ & $0$ \\
\hline
$6$ & $7$ & $0$ & $0$ & $0$ & $0$ & $0$ \\
\hline
$7$ & $17$ & $4$ & $0$ & $0$ & $0$ & $0$ \\
\hline
$8$ & $59$ & $21$ & $13$ & $0$ & $0$ & $0$ \\
\hline
$9$ & $216$ & $127$ & $64$ & $0$ & $0$ & $0$ \\
\hline
\end{tabular}
\end{table}
\end{remark}

\section{Proofs of Theorems \ref{thm:main1} and \ref{thm:k2d-4}, and Corollaries~\ref{cor:gap} and \ref{cor:explicit-values-lower-dim}}
\label{sec:proof}

\begin{proof}[Proof of \cref{thm:main1}]
After replacing $f$ by an iterate (see \cite[Lemma~2.6]{Hu-GK-AV}), we may work under the notation and assumptions of \cref{assum:XfN}.
Denote the associated dynamical intersection polynomial by $\Phi_{f, H_X}(\bx)$; see \cref{def:Phi}.
It follows readily from \cref{thm:plov=d+max} and \cref{prop:thm3.2}\ref{prop:thm3.2-2} that
\begin{align}\label{eq:main1}
\plov(f) \ge d + |\kappa(\bt)| = d + \sum_{j=1}^r2jt_j \ge d + r(r+1) = d + \frac{k(k+2)}{4},
\end{align}
where the second inequality holds since each $t_j\ge 1$.

The optimality of the inequality \eqref{eq:main1} follows from \cref{lem:possible-plov}.
Indeed, write $r=k/2$ and choose
\[
\mu=(r+1,1^{d-r-1},0^{r})\in P(d,d,d).
\]
Then
\[
\sum_{i=1}^d \mu_i^2 = (r+1)^2 + d-r-1 = d+r(r+1) = d+k(k+2)/4.
\]
Thus the lower bound is realizable by an abelian variety of dimension $d$.

For the last equivalence statement, recall that $k$ is even with $k\leq 2d-2$ and $\plov(f)\leq (k/2+1)d$ by \cref{thm:upper-bounds}.
Combining with \cref{eq:main1}, we get that $\plov(f) =  d^2$ if and only if $k=2d-2$.
\end{proof}

\begin{proof}[Proof of \cref{thm:k2d-4}]
After replacing $f$ by an iterate (see \cite[Lemma~2.6]{Hu-GK-AV}), we may work under the notation and assumptions of \cref{assum:XfN}.
Recall that $k=2r\le 2d-4$ by assumption.
If $r\le d-3$, then \cref{thm:upper-bounds} immediately gives
\[\plov(f) \le (r+1)d \le d(d-2).\]
So, we only need to consider the case when $r=d-2$.
Let $t_0,t_1,\ldots,t_r$ be the positive integers given by \cref{prop:thm3.2}, so that $\sum_{j=0}^r t_j = d$.
Since each $t_j$ is positive and $r+1=d-1$, it follows that exactly one of the $t_j$ is equal to $2$, while all the others are equal to $1$. Thus,
\[
\bt \coloneqq (t_0,t_1,\ldots,t_r) = (1,\ldots,1,2,1,\ldots,1),
\]
where $t_{j_\circ}=2$ for some $j_\circ \in \bZ \cap [0,r]$.
The corresponding restricted partition $\kappa(\bt)$ is
\[
\kappa(\bt) \coloneqq (2r,2r-2,\ldots,2j_\circ,2j_\circ,\ldots,2,0) \in P\bigl(2r,r+2,r(r+1)+2j_\circ\bigr).
\]

\begin{claim}\label{claim:lower-bound-on-e}
One has $w_\lambda=0$ for all $\lambda\in P(2r,r+2,r(r+1)+e)$ and all integers
\[
e\ge \max\bigl\{\, \lfloor r/2\rfloor, \, j_\circ+1 \,\bigr\}.
\]
\end{claim}
\begin{proof}[Proof of Claim \ref{claim:lower-bound-on-e}]
Fix such an integer $e$.
By \cref{prop:thm3.2}\ref{prop:thm3.2-1}, we already know that $w_\lambda=0$ for all $\lambda\in P(2r,r+2,n)$ satisfying $\lambda \succ \kappa(\bt)$.
Therefore, it remains to show that $w_\lambda=0$ for all $\lambda\in P(2r,r+2,r(r+1)+e)$ with $\lambda \preceq \kappa(\bt)$.
By \cref{lemma:equations-v-lambda}, these numbers form a vector $(w_\lambda)_{\lambda \preceq \kappa(\bt)}$ lying in the nullspace of the truncated matrix $A_{2r,r+2,r(r+1)+e}^{\preceq \kappa(\bt)}$ as defined in \cref{def:truncated-matrix}.
Applying \cref{thm:nullity-truncated-matrix2} with $\ell=j_\circ$, we obtain
\[
\nullity\bigl(A_{2r,r+2,r(r+1)+e}^{\preceq \kappa(\bt)}\bigr)=0
\]
since $e\ge \max\bigl\{\, \lfloor r/2\rfloor, \, j_\circ+1 \,\bigr\}$.
This proves the claim.
\renewcommand{\qedsymbol}{}
\end{proof}

Now recall from \cref{prop:thm3.2}\ref{prop:thm3.2-2} that $w_{\kappa(\bt)}>0$.
Since $\kappa(\bt) \in P\bigl(2r,r+2,r(r+1)+2j_\circ\bigr)$, \cref{claim:lower-bound-on-e} implies that
\begin{align}
\label{eq:j0}
2j_\circ < \max\bigl\{\, \lfloor r/2\rfloor, \, j_\circ+1 \,\bigr\}.
\end{align}
Recall that $d\ge 4$ by assumption, so that $r=d-2\ge 2$. If $\lfloor r/2\rfloor<j_\circ+1$, then \cref{eq:j0} becomes $2j_\circ<j_\circ+1$, i.e., $j_\circ=0$. But $j_\circ+1>\lfloor r/2\rfloor\ge 1$, a contradiction.
So we indeed have $\max\bigl\{\, \lfloor r/2\rfloor, \, j_\circ+1 \,\bigr\} = \lfloor r/2\rfloor$.
It thus follows from \cref{claim:lower-bound-on-e} and \cref{thm:plov=d+max} that
\[
\plov(f) \le d + r(r+1) + \lfloor r/2\rfloor - 1 = d(d-2) + \lfloor d/2 \rfloor.
\]

Finally, \cref{prop:same-parity} gives $\plov(f)\equiv d \pmod 2$.
Hence the upper bound $d(d-2) + \lfloor d/2 \rfloor$ cannot be attained when $\lfloor d/2 \rfloor$ is odd, i.e., $d\equiv 2,3 \pmod 4$.
This completes the proof of \cref{thm:k2d-4}.
\end{proof}

\begin{remark}
From the proof of \cref{thm:k2d-4}, \cref{eq:j0} turns out to be $2j_\circ \le \lfloor r/2\rfloor-1$, i.e.,
\[
0\le j_\circ \le \left\lfloor\frac{\lfloor r/2 \rfloor-1}{2}\right\rfloor = \lfloor d/4 \rfloor - 1.
\]
This shows that the unique repeated part $2j_\circ$ in the distinguished restricted partition $\kappa(\bt)$ cannot be too large.
In particular, $j_\circ=0$ whenever $4\le d\le 7$.
When $d=3$, i.e., $r=1$, \cref{eq:j0} also infers that $j_\circ=0$.
However, our method cannot go beyond the lower bound on $e$ in \cref{claim:lower-bound-on-e} which inherits from \cref{thm:nullity-truncated-matrix2}; see \cref{rmk:trun-A-table}.
\end{remark}

\begin{proof}[Proof of \cref{cor:gap}]
By \cref{thm:main1,thm:k2d-4}, if $d\ge 4$, then
\[
\plov(f)\notin \bigl(d(d-2)+2\lfloor d/4\rfloor,\, d^2\bigr).
\]
The lower dimensional cases $d=2, 3$ were classified explicitly by Artin--Van den Bergh \cite[Theorem~1.7]{AVdB90} and Lin--Oguiso--Zhang \cite[Corollary~1.3]{LOZ25}, respectively.

Alternatively, if $d=2$, one immediately obtains $\plov(f) \notin (2,4)$ from the parity principle; see \cref{prop:same-parity}.
When $d=3$, if $k=4$ then $\plov(f)=9$ by \cref{thm:main1}; if $k\le 2$ then $\plov(f)\le 5$ by \cref{rmk:odd-d}.
Hence we get $\plov(f) \notin (5,9)$.
\end{proof}

\begin{proof}[Proof of \cref{cor:explicit-values-lower-dim}]
Combining \cref{thm:main1} with \cref{thm:upper-bounds}, we now have the following general bounds for $\plov(f)$:
\begin{align}\label{eq:plov-bounds}
d+\frac{k(k+2)}{4} \le \plov(f) \le (k/2+1)d.
\end{align}
Assume that $d=4$ and $k$ is a nonnegative even integer with $0 \le k \le 2d-2=6$.
Clearly, $k=0$ if and only if $\plov(f)=d$.

If $k=2$, then \cref{eq:plov-bounds} becomes $6\le \plov(f) \le 8$.
According to \cref{prop:same-parity}, $\plov(f)$ is even.
So, $\plov(f)=6$ or $8$.
By \cref{lem:possible-plov}, both values are realizable by abelian varieties,
corresponding to $\mu=(2,1,1,0)$ and $\mu=(2,2,0,0)$, respectively.

If $k=4$, then \cref{thm:k2d-4} asserts that $\plov(f) \le 10$. Combining this with \cref{eq:plov-bounds}, we obtain $\plov(f)=10$, which is also realizable by \cref{lem:possible-plov}, corresponding to $\mu=(3,1,0,0)$.

If $k=6$ is maximal, then \cref{thm:main1} shows that $\plov(f)=4^2$, which is also realizable.
\end{proof}

To conclude, in view of \cref{lem:possible-plov}, we ask the following question.

\begin{question}[{cf.~\cite[Question~1.5(2)]{LOZ25}}]
\label{final question}
For every pair of integers $d\in \bZ_{>0}$ and $k\in 2\bZ \cap[0, 2d-2]$,  
are the following two sets the same:
\[
\Biggl\{\, \plov(f) \,:\,
\begin{aligned}
&  X \text{ is a normal projective variety}, \, \dim X= d,\\
& f\in \Aut(X), \, \deg_1(f^n) \asymp n^k
\end{aligned}
\,\Biggr\}
=
\Biggl\{\, \sum_{i=1}^{d} \mu_i^2 \,:\, 
\begin{aligned}
& \mu \in P(d,d,d), \\
& \mu_1=k/2+1
\end{aligned}
\,\Biggr\}?
\]
\end{question}

An affirmative answer to \cref{final question} will provide a very clear picture on the classification problem for $\plov$.
By \cref{lem:possible-plov}, it suffices to prove that the left-hand set is contained in the right-hand set.
The Second Gap Principle, i.e., \cref{conj:SGP}, is nothing but the case when $k=2d-4$ in \cref{final question}.
We could also ask what is the left-hand set for a given variety $X$ (e.g., a Calabi--Yau variety), but this might be much harder. 

In summary, \cref{final question} is already known to be true in the following cases:
\begin{itemize}
\item $k=0$ by the projection formula;
    \item $k=2$ by \cref{prop:same-parity,rmk:odd-d};
    \item $k=2d-2$ by \cref{thm:main1};
    \item $k=2d-4$ and $d\le 7$ by \cref{cor:gap} and \cref{rmk:suspected-gap};
    \item $d\le 4$ by \cite[Theorem~1.7]{AVdB90}, \cite[Corollary~1.3]{LOZ25}, and \cref{cor:explicit-values-lower-dim}.
\end{itemize}
However, it remains open even for $d=5$ and $k=4$.

\bibliographystyle{amsalpha}

\bibliography{mybib}

\end{document}